\documentclass[psamsfonts,fceqn,leqno]{amsart}
  \usepackage{amsthm}
  \usepackage{amsmath}
  \usepackage{amssymb}
  \usepackage{mathrsfs}
  
  \newtheorem{theorem}{Theorem}[section]

  \newtheorem{lemma}[theorem]{Lemma}

  \theoremstyle{definition}
  \newtheorem{definition}[theorem]{Definition}
  \newtheorem{remark}[theorem]{Remark}

  \numberwithin{equation}{section}

  \title[Infinite dimensional mixed economies with asymmetric
  information]
  {Infinite dimensional mixed economies\\ with asymmetric
  information}
  \author[A. Bhowmik]{Anuj Bhowmik}
  \address{School of Computing and Mathematical Sciences,
  Auckland University of Technology, Private Bag 92006, Auckland
  1142, New Zealand}
  \email{anuj.bhowmik@aut.ac.nz}

  \author[J. Cao]{Jiling Cao}
  \address{School of Computing and Mathematical Sciences,
  Auckland University of Technology, Private Bag 92006, Auckland
  1142, New Zealand}
  \email{jiling.cao@aut.ac.nz}

  \thanks{\hspace{-1.66em} \emph{JEL classification:} C71; D41;
  D43; D51; D82.}

  \date{}

\begin{document}

  \maketitle

  \begin{abstract}
  In this paper, we study asymmetric information economies
  consisting of both non-negligible and negligible agents
  and having ordered Banach spaces as their commodity spaces.
  In answering a question of Herv\'{e}s-Beloso and
  Moreno-Garc\'{i}a in \cite{Herves-Beloso-Moreno-Garcia:08},
  we establish a characterization of Walrasian expectations
  allocations by the veto power of the grand coalition. It is
  also shown that when an economy contains only negligible
  agents a Vind's type theorem on the private core with the
  exact feasibility can be restored. This solves a problem
  of Pesce in \cite{Pesce:10}.
  \end{abstract}

  \section{Introduction} \label{sec:intro}
  In their seminal papers \cite{Arrow-Debreu:54} and
  \cite{McKenzie:59}, Arrow, Debreu and McKenzie considered an
  economic model consisting of finitely many agents. Since only
  finitely many coalitions can be formed in such an economy,
  the characterization of Walrasian allocations by the veto
  mechanism is asymptotic \cite{Debreu-Scarf:63}. Later,
  Aumann \cite{Aumann:64} considered an economic model
  consisting of a continuum of agents by taking [0,1] with
  Lebesgue measure as the space of agents and established
  a characterization of Walrasian allocations in terms of
  the core. The main advantage of Aumann's model is that
  perfect competition prevails, that is, the influence of
  any individual agent on the economy is negligible. However,
  the competition in many real economies is imperfect, for
  instance, in an economy which has some individual agents who
  own large portions of initial endowments of some commodities.
  This is the main motivation to consider mixed economies or
  oligopolistic markets, refer to \cite{De Simone-Graziano:03},
  \cite{Greenberg-Shitovitz:86}, \cite{Pesce:10}, and
  \cite{Shitovitz:73}.
  In Chapter 7 of \cite{Debreu:59}, uncertainty was introduced
  in the Arrow-Debreu-McKenzie model by allowing finitely many
  states of nature and viewing the commodities as differentiated
  by state. In this model, each agent possesses the same full
  information and makes a contract contingent on the realized
  state of nature. However, such a model does not capture the
  idea of contracts under asymmetric information. This analysis
  was extended by Radner in \cite{Radner:68}, where each agent
  is characterized by a private information set, a
  state-dependent utility function, a random initial endowment
  and a prior belief. The trade of an agent is measurable with
  respect to his information so that he cannot act differently
  on states that he cannot distinguish and an agent makes a
  contract for trading commodities before he obtains any
  information about the realized state of nature. Radner also
  extended the notion of a Walrasian equilibrium in the
  Arrow-Debreu-McKenzie model to that of a Walrasian expectations
  equilibrium in his model so that better informed agents are
  generally better off.

  In this paper, we consider a mixed economy with asymmetric
  information and infinitely many commodities. In Section
  \ref{sec:model}, we provide a general description on our model.
  Section \ref{sec:atomless} is devoted to study a special case
  of our model, where the space of agents is an atomless measure
  space. Two results on the private blocking power of a coalition
  are established, and measures of blocking coalitions when
  agents are asymmetrically informed are studied. Schmeidler
  \cite{Schmeidler:72} first improved Aumann's equivalence
  result by only considering the blocking power of small
  coalitions in a complete information and atomless economy
  with finitely many commodities. Schmeidler's result was further
  generalized in Grodal \cite{Grodal:72}. Finally, Vind \cite{Vind:72}
  showed that if some coalition blocks an allocation then there
  is a blocking coalition with any measure less than the measure
  of the grand coalition. Although Herv\'{e}s-Beloso et al.
  \cite{Herves-Beloso-Moreno-Garcia-Nunez-Sanz-Pascoa:00}
  pointed out that analogous results of Vind's theorem are
  generally false for an atomless economy with the space of
  real bounded sequences as the commodity space, extensions of
  Vind's theorem for special economies with asymmetric
  information and the free disposal condition can be found in
  \cite{Bhowmik-Cao:submitted},
  \cite{Herves-Beloso-Moreno-Garcia-Yannelis:05a} and
  \cite{Herves-Beloso-Moreno-Garcia-Yannelis:05b}. Recently,
  Herv\'{e}s-Beloso et al. \cite{Herves-Beloso:prepint} established
  a Vind's type theorem for the process of information shared by
  coalitions in an asymmetric information economy having a finite
  dimensional commodity space and the free disposal assumption.
  Considering an ordered Banach space whose positive cone admitting
  an interior point as the commodity space and a complete finite
  positive atomless measure space of agents, Evren and
  H\"{u}sseinov \cite{Evern-Husseinov:08} established a Vind's
  type result on the private core of an economy under the free
  disposal condition and other additional assumptions.
  However, as mentioned in \cite{Pesce:10}, whether there is a
  version of Vind's theorem on the private core of an economy
  with the exact feasibility for finite dimensional economies
  is still an open problem. Here, we investigate this question
  for an asymmetric information economy with an ordered Banach
  space whose positive cone has an interior point as the
  commodity space and give a full solution. As a result, the
  equivalence theorem for finite dimensional economies in
  \cite{angeloni:09} is further generalized. The corresponding
  problems on the (strong) fine core of an economy are also
  considered.

  Concerning a complete information economy, Herv\'{e}s-Beloso
  and Mareno-Garc\'{i}a \cite{Herves-Beloso-Moreno-Garcia:08}
  provided a characterization of Walrasian allocations by
  robustly efficient allocations when the economy has a continuum
  of agents and finitely many commodities. More precisely, if
  $f$ is a Walrasian allocation then it is non-dominated in not
  only the initial economy but also all economies obtained by
  modifying the initial endowments of any coalition in the
  direction of $f$. In the same paper, they also showed that
  such a result holds for economies with asymmetric information
  and the space of real bounded sequences as the commodity space.
  In Section \ref{sec:mixed}, a similar result is established
  in an asymmetric information economy whose space of agents
  is a complete finite positive measure space and commodity space
  is an ordered separable Banach space whose positive cone has
  an interior point. Other results in Section \ref{sec:mixed}
  concern the relationships among different types of cores.
  Einy et al. \cite{Einy-Moreno-Shitovitz:00} showed that the
  fine core is a subset of the ex-post core for an asymmetric
  information economy with an atomless measure space of agents
  and a finite dimensional commodity space. One year later,
  they established a characterization of the weak fine core by
  the private core in a complete information economy in
  \cite{Einy-Moreno-Shitovitz:01}, where it was assumed that
  the grand coalition is a finite union of pairwise disjoint
  measurable subsets having positive measure and any
  two agents in the same measurable subset have the same
  information. Here, these results are extend to mixed
  economies with asymmetric information and ordered separable
  Banach spaces whose positive cones contain interior points as
  commodity spaces. Furthermore, in our framework there may exist
  an information type associated with a null measurable subset of
  the grand coalition.

  \section{The model} \label{sec:model}
  Let $\mathcal{E}$ be an exchange economy with asymmetric
  information as in \cite{Radner:68} and \cite{Radner:82}. Suppose
  that $(\Omega, \mathcal{F})$ is a measurable space, where $\Omega$
  is a finite set denoting all possible states of nature and the
  $\sigma$-algebra $\mathcal{F}$ denotes all events. Following
  from the well-known mixed market model, the space of agents is
  a measure space $(T,\Sigma,\mu)$ with a complete, finite and
  positive measure $\mu$, where $T$ is the set of agents, $\Sigma$
  is the $\sigma$-algebra of measurable subsets of $T$ whose economic
  weights on the market are given by $\mu$. Following from a
  classical result in measure theory, $T$ can be decomposed into
  two parts: one is atomelss and the other contains countably
  many atoms. That is, $T=T_0\cup T_1$, where $T_0$ is the atomless
  part and $T_1$ is the countable union of $\mu$-atoms. Since each
  $\mu$-atom is treated as an agent, $A\in T_1$ is used instead
  of $A\subseteq T_1$ if $A$ is a $\mu$-atom. Agents in $T_0$ are
  called ``\emph{small agents}" and those in $T_1$ are called
  ``\emph{large agents}". In each state, infinitely many
  commodities are assumed. Throughout, the commodity space of
  $\mathcal{E}$ is an ordered Banach space $Y$ whose positive
  cone has an interior point. The order on $Y$ is denoted by
  $\leq$, and $Y_+ =\{x\in Y: x \geq 0\}$ denotes the positive
  cone of $Y$. The symbol $x\gg 0$ (resp. $x> 0$) denotes a
  strictly positive (resp. non-zero positive) element $x$ of $Y_+$.
  The economy extends over two time periods $\tau = 0, 1$.
  Consumption takes place at $\tau= 1$. At $\tau= 0$, there is
  uncertainty over the states and agents make contracts that are
  contingent on the realized state at $\tau= 1$. Thus,
  ${\mathcal E}$ can be defined by
  \[
  {\mathcal E} = \{(\Omega, \mathcal{F});\ (T,\Sigma,\mu);\ Y_+;\
  ({\mathcal F}_t, U_t, a(t, \cdot), q_t)_{t\in T}\}.
  \]
  Here, $Y_+$ is the \emph{consumption set} in every state
  $\omega \in \Omega$ for every agent $t\in T$; ${\mathcal F}_t$
  the $\sigma$-algebra generated by a partition $\Pi_t$ of
  $\Omega$ representing the \emph{private information}
  of agent $t$; $U_{t}:\Omega \times Y_+\to\mathbb R$ is the
  \emph{state-dependent utility function} of agent $t$; $a(t,
  \cdot): \Omega\rightarrow Y_+$ is the \emph{random initial
  endowment} of agent $t$, assumed to be constant on elements
  of $\Pi_t$; and $q_t$ is a probability measure on $\Omega$
  giving the \emph{prior} of agent $t$. It is assumed that $q_t$
  is positive on all elements of $\Omega$. The quadruple
  $({\mathcal F}_t, U_t, a(t, \cdot), q_t)$ is called the
  \emph{characteristics} of the agent $t\in T$. A function $x:
  \Omega\to Y_+$ is interpreted as a random consumption bundle
  in $\mathcal E$. The \emph{ex ante expected utility} of an
  agent $t$ for a given random consumption bundle $x$ is defined
  by $V_{t}(x)= \sum_{\omega\in \Omega} U_{t} (\omega,x)q_t(\omega)$.

  Any set $S \in \Sigma$ with $\mu(S)> 0$ is called a \emph{coalition}
  of $\mathcal E$.  If $S$ and $S'$ are two coalitions of $\mathcal
  E$ with $S' \subseteq S$, then $S'$ is called a \emph{sub-coalition}
  of $S$. For a coalition $S$ in $\mathcal E$, an
  \emph{$S$-assignment} in $\mathcal E$ is a function $f: S\times
  \Omega\rightarrow Y_+$ such that $f(\cdot, \omega)\in L_1^S (\mu,
  Y_+)$ for all $\omega \in \Omega$, where $L_1^S (\mu, Y_+)$ is
  the set of all Bochner integrable functions from $S$ into $Y_+$.
  It is assumed that $a(\cdot,\omega)\in L_1^T(\mu, Y_+)$ for
  each $\omega \in \Omega$. Put $L_t = \{ x\in (Y_+)^\Omega:
  x \mbox{ is }\mathcal{F}_{t} \mbox{-measurable}\}$. An $
  S$-assignment $f$ in $\mathcal{E}$ is called an
  \emph{$S$-allocation} if $f(t,\cdot)\in L_t$ for almost all
  $t\in S$, and it is said to be \emph{$S$-feasible} if $\int_S
  f(\cdot, \omega) d\mu\leq \int_S a(\cdot, \omega)d\mu$ for all
  $\omega \in \Omega$. $T$-assignments, $T$-allocations and
  $T$-feasible allocations are simply called \emph{assignments},
  \emph{allocations} and \emph{feasible allocations}. A coalition
  $S$ \emph{privately blocks an allocation} $f$ in $\mathcal E$
  if there is an $S$-feasible allocation $g$ such that $V_t(g(t,
  \cdot))> V_t(f(t,\cdot))$ for almost all $t\in S$. The
  \emph{private core} of $\mathcal E$ is the set of all feasible
  allocations which are not privately blocked by any coalition.
  A \emph{price system} is an $\mathcal{F}$-measurable,
  non-zero function $\pi:\Omega\to Y_+^\ast$, where $Y_+^\ast$
  is the positive cone of the norm-dual space $Y^\ast$ of $Y$.
  The \emph{budget set} of agent $t$ can be defined by
  \[
  B_{t}(\pi) = \left\{ x\in L_t:\sum_{\omega\in
  \Omega} \langle\pi(\omega), x(\omega)\rangle\leq
  \sum_{\omega\in \Omega}\langle\pi(\omega), a(t, \omega)
  \rangle\right\}.
  \]
  A \emph{Walrasian expectations equilibrium} of
  $\mathcal{E}$ in the sense of Radner  is a pair $(f, \pi)$,
  where $f$ is a feasible allocation and $\pi$ is a price
  system such that for almost all $t\in T$, $f(t, \cdot)\in B_t
  (\pi)$ and $f(t, \cdot)$ maximizes $V_t$ on $B_t(\pi)$, and
  \[
  \sum_{\omega\in \Omega} \left \langle\pi(\omega),\
  \int_T f(\cdot, \omega) d\mu \right\rangle=
  \sum_{\omega\in \Omega} \left \langle\pi(\omega),\
  \int_T a(\cdot, \omega)d\mu \right\rangle.
  \]
  Two agents are said to be the \emph{same type} if they
  have the same characteristics. The family of partitions of
  $\Omega$ is denoted by $\mathfrak P$. For any $\mathcal Q\in
  \mathfrak P$, let $T_{\mathcal Q}= \{t\in T: \Pi_t=
  \mathcal Q\}$. For any coalition $S$, put $\mathfrak P_S =
  \{\mathcal Q\in \mathfrak P:S\cap T_\mathcal Q \neq \emptyset
  \}$ and $\mathfrak P(S)= \{\mathcal Q\in \mathfrak P_S: \mu(S
  \cap T_\mathcal Q)> 0\}$. Then,
  $\bigcup_{\mathcal Q\in \mathfrak P_T}T_{\mathcal Q}= T$ and
  $L_t= L_{t^\prime}$ if and only if $t, t^\prime\in T_{\mathcal
  Q}$ for some $\mathcal Q\in \mathfrak P_T$.
  For any $S \in \Sigma$, $\bigvee \mathfrak Q$ denotes the
  $\sigma$-algebra generated by the smallest refinement of all
  members of $\mathfrak Q\subseteq \mathfrak P_S$.

  \medskip
  \noindent
  {\bf Assumptions:}

  \begin{itemize}
  \item[(A$_1$)] \emph{Measurability}: The functions $t\mapsto q_t$
  and $t\mapsto \mathcal{F}_t$ are measurable. This means that
  $\{t\in T: q_t\in A\} \in \Sigma$ for any Borel subset $A$ of
  $|\Omega|-1$ dimensional unit simplex, and $T_{\mathcal Q}\in
  \Sigma$ for all $\mathcal Q\in \mathfrak P$.

  \item[(A$_2$)] \emph{Carath\'{e}odory}: For each $\omega\in
  \Omega$, $(t, x)\mapsto U_t(\omega, x)$ is a Carath\'{e}odory
  function on $T\times Y_+$. This means that $U_{(\cdot)}(\omega,
  x)$ is measurable for all $(\omega, x)\in \Omega\times Y_+$,
  and $U_t(\omega, \cdot)$ is norm-continuous for all $(t, \omega)
  \in T\times \Omega$.

  \item[(A$_3$)] \emph{Monotonicity}: For each $(t, \omega) \in T
  \times \Omega$, if $x, y\in Y_+$ with $y\gg 0$, then $U_t(\omega,
  x+y)>  U_t(\omega, x)$.

  \item[(A$_3^\prime$)] \emph{Strong monotonicity}: For each $(t,
  \omega) \in T\times \Omega$, if $x, y\in Y_+$ with $y> 0$, then
  $U_t(\omega, x+y)> U_t(\omega, x)$.

  \item[(A$_4$)] \emph{Partial concavity}: For each $(t_0,\omega_0)
  \in T_1 \times \Omega$ and $S$-feasible assignment $f$
  with $\mu(S\cap T_1)> 0$ in $\mathcal E$, $U_{t_0}(\omega_0,
  \hat f(\omega_0))\ge \frac{1}{\mu(S\cap T_1)}\int_{S\cap T_1}
  U_{t_0}(\omega_0, f(\cdot, \omega_0))d\mu$, where $\hat
  f(\omega_0)= \frac{1}{\mu(S\cap T_1)}\int_{S\cap T_1}f(\cdot,
  \omega_0)d\mu$.

  \item[(A$_4^\prime$)]\emph{Concavity}: For each $(t, \omega)
  \in T\times \Omega$, $U_t(\omega, \cdot)$ is concave.

  \item[(A$_5$)] \emph{Strict positivity}: For each $(t,\omega)
  \in T\times \Omega$, $a(t, \omega)\gg 0$.

  \item[(A$_6$)] \emph{Similarity}: All large agents are the
  same type.

  \item[(A$_7$)] \emph{Minimality}: $T_1$ contains at least two
  elements.

  \item[(A$_8$)] \emph{Informativeness}: $\bigvee \mathfrak P_T=
  \mathcal F$.

  \item[(A$_9$)] \emph{$\mathcal F$-measurability}: For almost
  all $t\in T$ and $x\in Y_+$, $U_t(\cdot, x)$ is $\mathcal
  F$-measurable.
  \end{itemize}
  Under (A$_1$) and (A$_2$), $V_{(\cdot)}(\cdot): T \times
  (Y_+)^\Omega \to \mathbb R$ is a Carath\'{e}odory function.
  Under (A$_3$) (resp. (A$_3^\prime$)), $V_t$ is monotone (resp.
  strongly monotone) in the sense that if $x, y\in (Y_+)^\Omega$
  with $y(\omega)\gg 0$ (resp. $y(\omega)> 0$) for some $\omega
  \in \Omega$, then $V_t(x+y)> V_t(x)$. Clearly, (A$_4$) implies
  that $V_{t_0}$ is partially concave for all $t_0\in T_1$,
  that is, $V_{t_0}(\hat{f}(\cdot)) \ge \frac{1}{\mu(S\cap T_1)}
  \int_{S \cap T_1}V_{t_0}(f(\cdot,\cdot))d\mu$ for all $t_0
  \in T_1$ and $S$-feasible assignment $f$ in $\mathcal E$
  with $\mu(S\cap T_1)> 0$, where $\hat{f}$ is defined in (A$_4$).
  Similarly, (A$_4^\prime$) implies that $V_t$ is
  concave for all $t\in T$. By (A$_6$), all agents in $T_1$ have
  the same characteristics, so we use $(\mathcal F_{T_1},
  U_{T_1}, a (T_1, \cdot), q_{T_1})$ to denote their common
  characteristics. Similarly, $V_{T_1}$ denotes the common ex
  ante expected utility of agents in $T_1$. Note that (A$_8$)
  is similar to $(A.4)$ in \cite{Einy-Moreno-Shitovitz:00}, and
  (A$_1$)-(A$_3$), (A$_5$) are the same as those in
  \cite{Evern-Husseinov:08}. For undefined mathematical concepts
  and terminologies in this paper, refer to \cite{Aliprantis-Border:06}.

  \section{Privately blocking and exact feasibility\\ in
  atomless economies} \label{sec:atomless}

  In this section, we study privately blocking and exactly feasible
  allocations in an atomless economy. Thus, we assume $T= T_0$ in
  this section. Two lemmas are established in Subsection
  \ref{subsec:pribck}, which are used in Section \ref{sec:mixed}.
  Similar to that in \cite{Vind:72}, we also investigate the
  blocking power of a coalition for the (strong) fine core and
  the private core when the exact feasibility is imposed on allocations.

  \subsection{Privately blocking coalitions} \label{subsec:pribck}

  The following result is similar to Lemma 1 in
  \cite{Evern-Husseinov:08}. In order to obtain a slightly different
  conclusion, we provide a proof here.

  \begin{lemma}\label{lem: block 1}
  Let an allocation $f$ in $\mathcal E$ be privately blocked
  by a coalition $S$ and $\alpha\in (0, 1)$. Under \emph{(A$_1$),
  (A$_2$)} and \emph{(A$_5$)}, there exist an $S$-allocation $g$
  and a sub-coalition $S^\prime$ of $S$ such that
  \begin{itemize}
  \item[{\rm(i)}] $g(t, \omega)\gg 0$ for all $(t,\omega) \in
  S^\prime \times \Omega$, and $V_t(g(t, \cdot))> V_t(f(t,
  \cdot))$ for almost all $t\in S$
  \item[(ii)] $\int_S (a(\cdot, \omega)- g(\cdot, \omega)) d\mu
  \gg 0$ for all $\omega\in \Omega$,
  \item[{\rm(iii)}] $\mu(S^\prime\cap T_\mathcal Q)>
  \alpha \mu(S\cap T_\mathcal Q)$ for all $\mathcal Q \in
  \mathfrak P(S)$.
  \end{itemize}
  \end{lemma}

  \begin{proof}
  Since $f$ is privately blocked by the coalition $S$, there
  exists an $S$-feasible allocation $h$ such that $V_{t}(h(t,
  \cdot))> V_{t}(f(t,\cdot))$ for almost all $t\in S$. Define
  a correspondence $P_f:S\rightrightarrows
  (Y_+)^\Omega$ by $P_f(t)=\{y\in L_t: V_t(y)> V_t(f(t,
  \cdot))\}$ for each $t\in S$. Then $h(t, \cdot)\in P_f(t)$ for
  almost all $t\in S$. By ignoring a $\mu$-null subset of $S$,
  one can choose a separable, closed linear subspace $Z$ of
  $Y^\Omega$ such that $f(S,\cdot) \cup h(S, \cdot) \cup a(S,
  \cdot)\subseteq Z$. Consider a correspondence
  $\widetilde{P}_f:S\rightrightarrows Z$
  defined by $\widetilde{P}_f(t)= Z\cap P_f(t)$. By Remark 6 in
  \cite{Evern-Husseinov:08}, Gr$_{\widetilde{P}_f}\in \Sigma_S
  \otimes \mathfrak{B}(Z)$, where $\Sigma_S =\{A \in \Sigma:
  A \subseteq S\}$, Gr$_{\widetilde{P}_f}$ denotes the graph
  of ${\widetilde{P}_f}$ and $\mathfrak{B}(Z)$ is the family of
  Borel subsets of $Z$. For any $\epsilon> 0$, define a
  correspondence $N_{\epsilon}:S\rightrightarrows Z$ by
  $N_{\epsilon}(t)= \{y\in Z: \|y- h(t, \cdot)\|<\epsilon\}$.
  Then, Gr$_{N_\epsilon}\in \Sigma_S\otimes \mathfrak{B}(Z)$.
  Furthermore, Gr$_{\tilde L}\in \Sigma_S\otimes \mathfrak{B}
  (Z)$, where $\tilde L: S\rightrightarrows Z$ is defined by
  $\tilde L(t)= Z\cap L_t$. For all $t\in S$, choose $\epsilon_t$
  such that $\epsilon_t=\sup\{\epsilon>0: y\in \widetilde{P}_f(t)
  \mbox{ whenever } y\in \tilde{L}(t)\cap N_\epsilon(t)\}$.
  Continuity of $V_t$ implies $\epsilon_t> 0$ for almost all
  $t\in S$. Let $\beta> 0$. Then,
  \[
  \{t\in S: \epsilon_t< \beta\}= \bigcup_{r\in \mathbb{Q}\cap(0,
  \beta)}\{t\in S: N_{r}(t)\cap \tilde{L}(t)\cap (Z\setminus
  \widetilde{P}_f(t))\neq \emptyset\},
  \]
  which is the projection of the set
  \[
  \bigcup_{r\in \mathbb{Q}\cap(0, \beta)}\left({\rm Gr}_{N_{r}}
  \cap {\rm Gr}_{\tilde L}\cap \left(S\times Z\setminus
  {\rm Gr}_{\widetilde{P}_f}\right)\right)\in \Sigma_S\otimes
  \mathfrak B(Z)
  \]
  on $S$. By the projection theorem \cite[p.608]{Aliprantis-Border:06},
  the set $\{t\in S: \epsilon_t < \beta\}\in \Sigma$, which
  means that the function $t\mapsto \epsilon_t$ is measurable.
  Choose a sequence $\{c_m\}\subset (0, 1)$ such that $c_m
  \rightarrow 0$ as $m\to \infty$. For each $m\geq 1$, define
  $h_m: S\times \Omega\to Y_+$ such that $h_m(t,
  \omega)= (1-c_m) h(t, \omega)+ \frac{c_m}{2} a(t, \omega)$.
  Then, $h_m$ is an $S$-allocation, and $h_m(t, \omega)\gg 0$
  for all $(t, \omega)\in S\times \Omega$. For each $m\geq 1$,
  put $S_m= \{t\in S: \|h_m(t, \cdot)- h(t, \cdot)\|<
  \epsilon_t\}$. Clearly, $S_m\in \Sigma_S$ and $S_m
  \subseteq S_{m+1}$ for all $m\geq 1$. Moreover, $\bigcup_{m}
  S_m\sim S$ and hence, $\lim_{m\to \infty}\mu(S\setminus S_m)=
  0$. By the definition of $\epsilon_t$, $h_m(t,\cdot)\in
  P_{f}(t)$ for almost all $t\in S_m$. For each $m\geq 1$,
  we now define a function $g_m:S\times \Omega\to Y_+$ by
  \[
  g_m(t, \omega)= \left\{
  \begin{array}{ll}
  h(t, \omega), & \mbox{if $t\in (S\setminus S_m)\times
  \Omega$;}\\[0.5em]
  h_{m}(t, \omega), & \mbox{if $(t, \omega)\in S_m\times
  \Omega$.}
  \end{array}
  \right.
  \]
  Then $g_m$ is an $S$-allocation, $V_t(g_m(t, \cdot))> V_t(f(t,
  \cdot))$ for almost all $t\in S$, and $g_m(t, \omega)\gg 0$
  for all $(t, \omega)\in S_m\times \Omega$. Now, for each
  $\omega\in \Omega$,
  \begin{eqnarray*}
  \int_{S}g_m(\cdot, \omega)d \mu &=& \int_{S\setminus S_m}
  h(\cdot,\omega)d\mu+ \int_{S_m}h_{m}(\cdot, \omega)d \mu \\
  &=& \int_{S\setminus S_m}(h(\cdot, \omega)- h_m(\cdot, \omega))
  d \mu+\int_{S}h_{m}(\cdot, \omega)d \mu.
  \end{eqnarray*}
  In addition, $\int_S h_{m}(\cdot,\omega)d \mu\leq \left(1-
  \frac{c_m}{2}\right)\int_S a(\cdot, \omega)d \mu$. Consequently,
  we obtain
  \[
  \int_{S}g_m(\cdot,\omega)d \mu \leq \int_{S\setminus S_m}c_m
  \left(h(\cdot, \omega)-\frac{1}{2}a(\cdot, \omega)\right)d \mu
   +\left(1- \frac{c_m}{2}\right)\int_S a(\cdot, \omega)d \mu,
  \]
  which is equivalent to
  \[
  \int_{S}(a(\cdot, \omega)- g_{m}(\cdot, \omega))d \mu\geq c_m
  \left(\frac{1}{2}\int_S a(\cdot, \omega)d \mu - z_m(\omega)
  \right),
  \]
  where $z_m(\omega)= \int_{S\setminus S_m} \left(h(\cdot,\omega)-
  \frac{1}{2}a(\cdot,\omega)\right) d\mu$. Since $\int_S a(\cdot,
  \omega) d\mu\gg 0$, by absolute continuity of the Bochner
  integral, $\frac{1}{2}\int_S a(\cdot,\omega)d \mu- z_m(\omega)
  \gg 0$ for all $\omega\in\Omega$ when $m$ is sufficiently large.
  Pick a $\mathcal Q_0\in \mathfrak P(S)$ satisfying $\mu(S\cap
  T_{\mathcal Q_0})\le \mu(S\cap T_\mathcal Q)$ for all $\mathcal
  Q\in \mathfrak P(S)$ and select a $1<\delta < \frac{1}{\alpha}$.
  Then for all $\mathcal Q\in \mathfrak P(S)$,
  \begin{eqnarray} \label{eqn:q0}
  (1-\alpha \delta)\mu(S\cap T_{\mathcal Q_0}) <
  (1-\alpha) \mu(S\cap T_{\mathcal Q}).
  \end{eqnarray}
  Choose some integer $m$ such that $\mu(S_m)> \alpha \delta
  \mu(S\cap T_{\mathcal Q_0})+ \mu(S\setminus T_{\mathcal Q_0})$.
  Obviously, $\mu(S_m \cap T_{\mathcal Q_0}) > \alpha \mu
  (S\cap T_{\mathcal Q_0})$. It is claimed
  that $\mu(S_m \cap T_{\mathcal Q})\le \alpha \mu(S\cap
  T_{\mathcal Q})$ implies $(1- \alpha \delta) \mu(S\cap
  T_{\mathcal Q_0}) \ge (1- \alpha)\mu(S\cap T_{\mathcal Q})$
  for any $\mathcal Q\in \mathfrak P(S)\setminus \{\mathcal
  Q_0\}$. If not, there is some $\mathcal Q^\prime\in \mathfrak
  P(S)\setminus \{\mathcal Q_0\}$ such that
  $\mu(S_m \cap T_{\mathcal Q^\prime})\le \alpha \mu(S\cap
  T_{\mathcal Q^\prime})$ but $(1- \alpha \delta) \mu(S\cap
  T_{\mathcal Q_0}) < (1- \alpha)\mu(S\cap T_{\mathcal Q^\prime})$.
  It follows that
  \begin{eqnarray*}
  \mu(S_m) &=& \mu(S_m\cap T_{\mathcal Q^\prime})+ \sum_{\mathcal
  Q\in \mathfrak P(S)\setminus \{\mathcal Q^\prime \}}
  \mu(S_m\cap T_{\mathcal Q})\\
  &\le& \alpha \mu(S\cap T_{\mathcal Q^\prime})+ \mu(S\cap
  T_{\mathcal Q_0})+ \sum_{\mathcal Q\in \mathfrak P(S)\setminus
  \{\mathcal Q_0, \mathcal Q^\prime\}}\mu(S\cap T_{\mathcal Q})\\
  &<& \alpha \delta \mu(S\cap T_{\mathcal Q_0})+ \mu(S\setminus
  T_{\mathcal Q_0}),
  \end{eqnarray*}
  which contradicts with the choice of $S_m$. This verifies the
  claim. By (\ref{eqn:q0}) and the claim, we conclude that
  $\mu(S_m\cap T_\mathcal Q)>\alpha \mu(S\cap T_\mathcal Q)$ for
  all $\mathcal Q\in \mathfrak P(S)$. The proof is completed by
  letting $g= g_m$ and $S^\prime= S_m$.
  \end{proof}

  \begin{remark}
  The conclusion of Lemma \ref{lem: block 1} is also true if the
  atomless measure space is replaced by a complete finite positive
  measure space.
  \end{remark}

  \begin{lemma} \cite{uhl:69} \label{lem:lyapunov}
  Suppose that $(X, \Sigma, \mu)$ is an atomless measure space
  and $E$ is a Banach space. If $f\in L_1^X(\mu, E)$, then the
  set $H= {\rm cl} \{\int_B f: B\in \Sigma\}$ is convex.
  \end{lemma}

  The following result is an extension of the result used in the
  main theorem of \cite{Vind:72} to an asymmetric information
  economy whose commodity space is an ordered Banach space
  having an interior point in its positive cone.

  \begin{lemma}\label{lem: block 2}
  Let $f$ be an allocation in $\mathcal E$. Suppose there exist
  a coalition $S$, a sub-coalition $S^\prime$ of $S$ and an
  $S$-allocation $g$ such that $g(t, \omega)\gg 0$ for all
  $(t, \omega)\in S^\prime\times \Omega$, $\mathfrak P(S)=
  \mathfrak P(S^\prime)$ and $V_t(g(t, \cdot))> V_t(f(t,
  \cdot))$ for almost all $t\in S$. Under \emph{(A$_1$)-(A$_3$)},
  for each $r\in (0,1)$, there exists an $S$-allocation $h$ such
  that $V_t (h(t,\cdot))> V_t(f(t, \cdot))$ for almost all
  $t\in S$, and $\int_{S} h(\cdot,\omega)d \mu= \int_{S}(r
  g(\cdot,\omega)+ (1- r) f(\cdot, \omega))d\mu$ for all
  $\omega\in \Omega$.
  \end{lemma}

  \begin{proof}
  For each $m\geq 1$, let $g_m: S\times \Omega\to Y_+$ be
  defined by $g_m(t,\omega)= (1-c_m) g(t, \omega)$. Then $g_m$
  is an $S$-allocation and $g_m(t, \omega)\gg 0$ for all $(t,
  \omega)\in S^\prime \times \Omega$.
  Pick an $r\in (0, 1)$ and a $\mathcal Q\in \mathfrak P(S)$.
  Let $\{c_m\}$ be a sequence in $(0, 1)$ such that $c_m\to 0$
  as $m\to \infty$. Applying an argument similar to
  that in Lemma \ref{lem: block 1}, it can be shown that there
  is an increasing sequence $\{S_m^\mathcal Q\}\subseteq
  \Sigma_{S\cap T_\mathcal Q}$ such that $\bigcup_{m}S_m^\mathcal
  Q\sim S\cap T_\mathcal Q$, $\lim_{m \to \infty} \mu((S\cap
  T_\mathcal Q)\setminus S_m^\mathcal Q)= 0$, and $V_t(g_m(t,
  \cdot))> V_t(f(t, \cdot))$ for almost all $t\in S_m^\mathcal Q$.
  Choose an $m_\mathcal Q$ such that $\mu(S^\prime\cap T_\mathcal
  Q\cap S_{m_\mathcal Q}^\mathcal Q)> 0$. Consider the function
  $y^\mathcal Q:(S\cap T_\mathcal Q)\times \Omega\to Y_+$
  defined by
  \[
  y^\mathcal Q (t, \omega)= \left\{
  \begin{array}{ll}
  g_{m_\mathcal Q}(t, \omega), & \mbox{if $(t, \omega)\in
  S_{m_\mathcal Q}^\mathcal Q\times \Omega$;}\\[0.5em]
  g(t, \omega), & \mbox{otherwise.}
  \end{array}
  \right.
  \]
  Obviously, $y^\mathcal Q$ is an $(S\cap T_\mathcal Q)$-allocation,
  and $V_t(y^\mathcal Q (t, \cdot))> V_t(f(t, \cdot))$ for almost
  all $t\in S\cap T_\mathcal Q$. Furthermore, for all $\omega\in
  \Omega$,
  \[
  \int_{S\cap T_\mathcal Q} y^\mathcal Q(\cdot, \omega)d \mu=
  \int_{S\cap T_\mathcal Q} g(\cdot, \omega)d \mu- c_{m_\mathcal
  Q}\int_{S_{m_\mathcal Q}^\mathcal Q}g(\cdot, \omega)d \mu.
  \]
  Let $x^\mathcal Q\gg 0$ be chosen such that $x^\mathcal Q \le
  \frac{c_{m_\mathcal Q}}{2}\int_{S_{m_\mathcal Q}^\mathcal Q}
  g(\cdot, \omega)d \mu$ for all $\omega\in \Omega$. Let $U(r,
  \mathcal Q)$ be an open neighborhood of $0$ such that $r
  x^\mathcal Q- U(r, \mathcal Q)\subseteq$ int$Y_+$. By Lemma
  \ref{lem:lyapunov},
  \[
  H_\mathcal Q= {\rm cl} \left\{\left(\mu(E^\mathcal Q),
  \int_{E^\mathcal Q}(y^\mathcal Q-f) d \mu\right)\in\mathbb R
  \times Y^\Omega: E^\mathcal Q\in\Sigma_{S\cap T_\mathcal Q}
  \right\}
  \]
  is a convex set. So, there is a sequence $\{E_n^\mathcal Q\}
  \subseteq \Sigma_{S\cap T_\mathcal Q}$ such that for each
  $\omega \in \Omega$,
  \[
  \lim_{n \to \infty} \left(\mu(E_n^\mathcal Q),
  \int_{E_n^\mathcal Q} (y^\mathcal Q(\cdot, \omega)-
  f(\cdot, \omega))d \mu\right)= r \left(\mu(S\cap T_\mathcal Q),
  \ z^\mathcal Q(\omega)\right),
  \]
  where $z^\mathcal Q(\omega)= \int_{S\cap T_\mathcal Q} (y^\mathcal
  Q(\cdot,\omega)- f(\cdot,\omega))d \mu$. Define a function
  $b_n^\mathcal Q:\Omega\to Y$ such that $b_n^\mathcal Q(\omega)
  = \int_{E_n^\mathcal Q}(y^\mathcal Q(\cdot, \omega)-f(\cdot,
  \omega)) d\mu-r z^\mathcal Q(\omega)$. Since $\|b_n^\mathcal
  Q(\omega)\|\to 0$ as $n\to \infty$, there is an $n_\mathcal Q$
  such that $b_{n_\mathcal Q}^\mathcal Q(\omega)\in U(r, \mathcal
  Q)$ for all $\omega\in \Omega$ and $\mu(E_{n_\mathcal Q}^\mathcal
  Q)< \mu(S\cap T_\mathcal Q)$. Consider the function $g^\mathcal
  Q:(S\cap T_\mathcal Q)\times \Omega\to Y_+$ defined by
  \[
  g^\mathcal Q(t, \omega)= \left\{
  \begin{array}{ll}
  y^\mathcal Q(t, \omega), & \mbox{if $(t, \omega)\in E_{n_\mathcal
  Q}^\mathcal Q\times \Omega$;}\\[0.5em]
  f(t, \omega)+ \frac{rx^\mathcal Q}{\mu \left((S\cap T_\mathcal
  Q)\setminus E_{n_\mathcal Q}^\mathcal Q\right)}, & \mbox{if
  $(t, \omega)\in (\left(S\cap T_\mathcal Q)\setminus E_{n_\mathcal
  Q}^\mathcal Q\right)\times \Omega$.}
  \end{array}
  \right.
  \]
  By (A$_3$), we have $V_t(g^\mathcal Q(t, \cdot))> V_t(f(t, \cdot))$
  for almost all $t\in S\cap T_\mathcal Q$ and $g^\mathcal Q$ is an
  $(S\cap T_\mathcal Q)$-allocation. Thus, we have
  \begin{eqnarray*}
  \int_{S\cap T_\mathcal Q} g^\mathcal Q(\cdot, \omega)d \mu &=&
  \int_{E_{n_\mathcal Q}^\mathcal Q} (y^\mathcal Q(\cdot,\omega) -
  f(\cdot,\omega))d \mu+ \int_{S\cap T_\mathcal Q} f(\cdot,\omega)
  d \mu+ rx^\mathcal Q.
  \end{eqnarray*}
  Furthermore, for all $\omega\in \Omega$,
  \[
  \int_{E_{n_\mathcal Q}^\mathcal Q} (y^\mathcal Q(\cdot,\omega)-
  f(\cdot, \omega))d \mu- b_{n_\mathcal Q}^\mathcal Q(\omega)= r
  \int_{S\cap T_\mathcal Q} (y^\mathcal Q(\cdot,\omega)
  -f(\cdot, \omega))d \mu.
  \]
  Consequently, we obtain
  \[
  \int_{S\cap T_\mathcal Q} g^\mathcal Q(\cdot, \omega)d\mu \ll
  \int_{S\cap T_\mathcal Q}(r y^\mathcal Q(\cdot,\omega)+ (1- r)
  f(\cdot, \omega))d\mu+ 2r x^\mathcal Q
  \]
  for each $\omega\in \Omega$, which implies that for each $\omega
  \in \Omega$,
  \[
  \int_{S\cap T_\mathcal Q} g^\mathcal Q(\cdot, \omega)d\mu \ll
  \int_{S\cap T_\mathcal Q}(r g(\cdot,\omega)+ (1- r)f(\cdot,
  \omega))d\mu.
  \]
  We now define a $\mathcal Q$-measurable $d^\mathcal Q: \Omega
  \to Y_+$ such that for each $\omega\in \Omega$,
  \[
  d^\mathcal Q(\omega)= \frac{1}{\mu(S\cap T_\mathcal Q)}
  \left[\int_{S\cap T_\mathcal Q}(r g(\cdot,\omega)+ (1- r)
  f(\cdot, \omega))d\mu- \int_{S\cap T_\mathcal Q} g^\mathcal
  Q(\cdot, \omega)d\mu\right].
  \]
  Clearly, $d^\mathcal Q(\omega)\gg 0$ for each $\omega\in \Omega$.
  Define an $(S\cap T_\mathcal Q)$-allocation by $h^\mathcal
  Q(t,\omega)= g^\mathcal Q(t, \omega)+ d^\mathcal Q(\omega)$
  for all $(t,\omega)\in (S\cap T_\mathcal Q) \times \Omega$.
  Then, $V_t(h^\mathcal Q(t, \cdot))> V_t(f(t,\cdot))$ for almost
  all $t\in S\cap T_\mathcal Q$ and $\int_{S\cap T_\mathcal Q}
  h^\mathcal Q(\cdot, \omega)d\mu= \int_{S\cap T_\mathcal Q}(r
  g(\cdot,\omega)+ (1- r)f(\cdot, \omega))d\mu$ for all $\omega
  \in \Omega$. Let $h:S\times \Omega\to Y_+$ be defined by
  \[
  h(t, \omega) = \left\{
  \begin{array}{ll}
  h^\mathcal Q(t, \omega), & \mbox{if $(t, \omega)\in (S\cap
  T_\mathcal Q)\times \Omega$ \mbox{ and} $\mathcal Q\in
  \mathfrak P(S)$ ;}\\[0.5em]
  g(t, \omega), & \mbox{otherwise.}
  \end{array}
  \right.
  \]
  It can be readily checked that $h$ is the desired $S$-allocation.
  \end{proof}

  \subsection{Allocations with the exact feasibility}
  \label{subsec:thfidom}
  In this subsection, we provide a characterization of exactly feasible
  allocations of $\mathcal E$ that are not in various types of cores.
  Given a coalition $S$ of $\mathcal E$, an $S$-assignment $f$ in
  $\mathcal E$ is called \emph{$S$-exactly feasible}
  if $\int_S f(\cdot,\omega)d\mu = \int_S a(\cdot, \omega)d\mu$ for
  all $\omega \in \Omega$. For simplicity, $T$-exactly feasible
  assignment is just termed as exactly feasible assignment. An
  allocation $f$ in $\mathcal E$ is \emph{NY-strongly
  fine\footnote{NY is the abbreviation of Nicholas Yannelis. Here,
  we follow some idea of his definition in \cite{Yannelis:91}, to
  distinguish it from the concept of Wilson in \cite{Wilson:78}.}
  blocked by a coalition $S$} \cite{Yannelis:91} if there exist
  a sub-coalition $S_0$ and an $S$-exactly feasible assignment $g$
  such that $g(t, \cdot)$ is $\bigvee \mathfrak P_S$-measurable and
  $V_t(g(t, \cdot))\geq V_t(f(t,\cdot))$ for almost all $t\in S$,
  and $V_t(g(t, \cdot))> V_t(f(t,\cdot))$ for almost all $t\in S_0$.
  The \emph{NY-strong fine core} \cite{Yannelis:91} of $\mathcal E$
  is the set of exactly feasible allocations which are not
  $NY$-strongly fine blocked by any coalition of $\mathcal E$.

  \begin{lemma}\label{lem: block 3}
  Let an allocation $f$ in $\mathcal E$ be $NY$-strongly fine
  blocked by a  coalition $S$ of $\mathcal E$. Under
  \emph{(A$_1$)-(A$_2$)}, \emph{(A$_3^\prime$)} and \emph{(A$_5$)},
  there exist a sub-coalition $S^\prime$ of $S$ and an $S$-assignment
  $g$ such that
  \begin{itemize}
  \item[{\rm (i)}] $g(t, \cdot)$ is $\bigvee \mathfrak P_S$-measurable
  and $V_t(g(t, \cdot))> V_t(f(t, \cdot))$ for almost all $t\in S$,

  \item[{\rm (ii)}] $g(t, \omega)\gg 0$ for all $(t, \omega)
  \in S^\prime \times \Omega$,

  \item[{\rm (iii)}] $\int_S (a(\cdot,\omega)- g(\cdot, \omega))d\mu
  \gg 0$ for all $\omega\in \Omega$.
  \end{itemize}
  \end{lemma}

  \begin{proof}
  Since $f$ is $NY$-strongly fine blocked by $S$, there are a
  sub-coalition $S_0$ of $S$ and an $S$-exactly feasible assignment
  $y$ such that $y(t, \cdot)$ is $\bigvee \mathfrak P_S$-measurable
  and $V_t(y(t, \cdot))\geq V_t(f(t, \cdot))$ for almost all $t\in S$,
  and $V_t(y(t, \cdot))> V_t(f(t, \cdot))$ for almost all $t\in S_0$.
  Without loss of generality, we may assume that
  $\mu(S_0)< \mu(S)$. Otherwise, the argument will be similar to that
  in Lemma \ref{lem: block 1}. By (A$^\prime_3$) and the fact that
  $V_t(y(t, \cdot))> V_t(f(t, \cdot))$ for almost all $t\in S_0$,
  there exist an atom $A$ of $\bigvee \mathfrak P_S$ and a
  sub-coalition $S_1$ of $S_0$ such that $y(t, \omega)> 0$ for all
  $\omega\in A$ and almost all $t\in S_1$. Let $\{c_m\}$ be a
  sequence in $(0, 1)$ converging to $0$. For each $m\geq 1$, we
  define a function $y_m:S_1\times \Omega\to Y_+$ such that $y_m
  (t,\omega)=(1-c_m) y(t, \omega)$. Then $y_m(t, \cdot)$ is
  $\bigvee \mathfrak P_S$-measurable for almost all $t\in S_1$. By an
  argument similar to that in the proof of Lemma \ref{lem: block 1},
  it can be shown that there is a sub-coalition $S_m$ of $S_1$ such
  that $V_t(y_m(t, \cdot))> V_t(f(t, \cdot))$ for almost all $t\in
  S_m$. Note that the function $b:A\to Y_+$, defined by $b(\omega)
  = \frac{c_m}{2}\int_{S_m} y(\cdot, \omega)d\mu$, is $\bigvee
  \mathfrak P_S$-measurable. Define a function $\hat y: (S\setminus
  S_0)\times \Omega\to Y_+$ by
  \[
  \hat y(t, \omega) = \left\{
  \begin{array}{ll}
  y(t, \omega)+ \frac{b(\omega)}{\mu(S\setminus S_0)}, & \mbox{if
  $(t, \omega)\in (S\setminus S_0)\times A$;}\\[0.5em]
  y(t, \omega), & \mbox{otherwise.}
  \end{array}
  \right.
  \]
  Furthermore define another function $h: S\times \Omega \to Y_+$ by
  \[
  h(t, \omega) = \left\{
  \begin{array}{ll}
  \hat y(t, \omega), & \mbox{if $(t, \omega)\in (S\setminus S_0)
  \times \Omega$;}\\[0.5em]
  y(t, \omega), & \mbox{if $(t, \omega)\in (S_0\setminus S_m)
  \times \Omega$;}\\[0.5em]
  y_{m}(t, \omega), & \mbox{if $(t, \omega)\in S_m\times \Omega$.}
  \end{array}
  \right.
  \]
  Then, $\hat y(t, \cdot)$ is $\bigvee \mathfrak P_S$-measurable
  and by (A$^\prime_3$), $V_t(\hat y(t, \cdot))> V_t(f(t, \cdot))$
  for almost all $t\in S\setminus S_0$. It follows that $h(t, \cdot)$
  is $\bigvee \mathfrak P_S$-measurable and $V_t(h(t, \cdot))> V_t
  (f(t, \cdot))$ for almost all $t\in S$, and $\int_S h(\cdot,
  \omega)d\mu\leq \int_S a(\cdot, \omega)d\mu$ for each $\omega\in
  \Omega$. Next, for each $m\geq 1$, define a function $h_m:
  S\times \Omega\to Y_+$ by $h_m(t,\omega)= (1-c_m) h(t,
  \omega)+ \frac{c_m}{2} a(t, \omega)$. Clearly, $h_m(t, \cdot)$ is
  $\bigvee \mathfrak P_S$-measurable for almost all $t\in S$, and
  $h_m(t, \omega)\gg 0$ for all $(t,\omega)\in S \times \Omega$.
  Applying an argument similar to that in the proof of Lemma
  \ref{lem: block 1}, one can find an increasing sequence $\{R_m\}
  \subseteq \Sigma_S$ such that $\bigcup_m R_m\sim S$ and $V_t
  (h_m(t, \cdot))> V_t(f(t, \cdot))$ for almost all $t\in R_m$.
  Finally, for each $m\ge 1$, consider the function $g_m: S\times
  \Omega\to Y_+$ defined by
  \[
  g_m(t, \omega) = \left\{
  \begin{array}{ll}
  h(t, \omega), & \mbox{if $(t,\omega) \in (S\setminus R_m)\times
  \Omega$;}\\[0.5em]
  h_{m}(t, \omega), & \mbox{if $(t, \omega)\in R_m\times \Omega$.}
  \end{array}
  \right.
  \]
  Following from the steps at the end of the proof of Lemma
  \ref{lem: block 1}, it can be verified that the conclusion of
  this lemma is true when $m$ is sufficiently large. Hence, the
  proof is completed by selecting such an $m$ and setting $S^\prime
  = R_m$ and $g= g_m$.
  \end{proof}

  \begin{theorem}\label{thm:finedomonated}
  Let an exactly feasible allocation $f$ be not in the $NY$-strong
  fine core of $\mathcal E$. Under \emph{(A$_1$)-(A$_2$)},
  \emph{(A$_3^\prime$)-(A$_4^\prime$)} and \emph{(A$_5$)}, for any
  $0< \epsilon< \mu(T)$, there is a coalition $S$ with $\mu(S)=
  \epsilon$ which $NY$-strongly fine blocks $f$.
  \end{theorem}

  \begin{proof}
  Suppose that $f$ is $NY$-strongly fine blocked by a coalition $S$.
  By Lemma $\ref{lem: block 3}$, there are a sub-coalition $S^\prime$
  of $S$ and an $S$-assignment $g$ such that (i)-(iii)
  of Lemma \ref{lem: block 3} hold. Define a function $z:\Omega\to
  Y_+$ such that  for all $\omega\in \Omega$,
  \begin{eqnarray} \label{eqn:zomega}
  z(\omega)= \int_{S}(a(\cdot, \omega)-g(\cdot, \omega))d\mu.
  \end{eqnarray}
  Then $z(\omega)\gg 0$ for all $\omega\in \Omega$. For any fixed
  $\mathcal Q\in \mathfrak P(S)$, by Lemma \ref{lem:lyapunov},
  \[
  H_{\mathcal Q}= {\rm cl}\left\{\left(\mu(E^{\mathcal Q}),
  \int_{E^{\mathcal Q}}(a-g) d \mu\right)\in \mathbb R\times
  Y^\Omega: E^{\mathcal Q}\in \Sigma_{S\cap T_\mathcal Q}\right\}
  \]
  is convex. For any given $\delta \in (0, 1)$, there
  is a sequence $\{E_n^\mathcal Q\}\subseteq \Sigma_{S\cap
  T_{\mathcal Q}}$ such that
  \[
  \lim_{n\to \infty} \left(\mu(E_n^\mathcal Q), \int_{E_n^\mathcal Q}
  (a(\cdot, \omega)-g(\cdot,\omega))d\mu\right)=\delta (\mu(S\cap
  T_{\mathcal Q}), \ z^\mathcal Q(\omega))
  \]
  for all $\omega\in \Omega$, where $z^\mathcal Q(\omega)=
  \int_{S\cap T_{\mathcal Q}}(a(\cdot, \omega) -g(\cdot,\omega))d\mu$.
  Since $\mu$ is atomless, we can select a sequence
  $\{F_n^\mathcal Q\}\subseteq \Sigma_{S\cap T_{\mathcal Q}}$
  such that $\mu(F_n^\mathcal Q)= \delta \mu(S\cap T_{\mathcal Q})$
  and $\mu(F_n^\mathcal Q\Delta E_n^\mathcal Q)=|\delta \mu(S\cap
  T_{\mathcal Q})- \mu(E_n^\mathcal Q)|$. Indeed, if
  $\mu(E_n^\mathcal Q)\ge \delta \mu(S\cap T_{\mathcal Q})$, we
  select any $F_n^\mathcal Q\subseteq E_n^\mathcal Q$ with
  $\mu(F_n^\mathcal Q)= \delta \mu(S\cap T_{\mathcal Q})$; Otherwise,
  we first select $C_n^\mathcal Q\subseteq (S\cap T_\mathcal Q)
  \setminus E_n^\mathcal Q$ with $\mu(C_n^\mathcal Q)= \delta
  \mu(S\cap T_\mathcal Q)- \mu(E_n^\mathcal Q)$ and put
  $F_n^\mathcal Q=E_n^\mathcal Q\cup C_n^\mathcal Q$. As a result,
  $\lim_{n\to \infty} \mu(F_n^\mathcal Q\Delta E_n^\mathcal Q)=0$,
  which implies that $\lim_{n\to \infty}
  \int_{F_n^\mathcal Q}(a(\cdot,\omega)-g(\cdot,\omega)) d\mu=
  \delta z^\mathcal Q(\omega)$ for all $\omega \in \Omega$. Let
  \[
  F_n= \left(\bigcup_{\mathcal Q\in \mathfrak P(S)} F_n^\mathcal Q
  \right)\bigcup \left(\bigcup_{\mathcal Q\in \mathfrak P_S\setminus
  \mathfrak P(S)}(S\cap T_\mathcal Q)\right)
  \]
  for all $n\in \mathbb N$. Then $\mu(F_n)= \delta \mu(S)$ and
  $\lim_{n \to \infty} \int_{F_n}(a(\cdot,\omega)-g(\cdot,\omega))
  d\mu= \delta z(\omega)$ for all $\omega \in \Omega$. Hence there
  is an $n_0$ such that $ \int_{F_{n_0}} (a(\cdot,\omega)-g(\cdot,
  \omega))d\mu\gg 0$ for all $\omega \in \Omega$. Since
  $\bigvee \mathfrak P_{F_{n_0}}=\bigvee \mathfrak P_S$, the
  function $z_{n_0}: \Omega\to Y_+$, defined by $z_{n_0}(\omega)=
  \int_{F_{n_0}} (a(\cdot, \omega)-g(\cdot,\omega))d\mu$,
  is $\bigvee \mathfrak P_{F_{n_0}}$-measurable. Define a
  function $\hat{g}:F_{n_0} \times \Omega\to Y_+$ such that
  $\hat{g}(t, \omega)= g(t, \omega)+ \frac{z_{n_0}(\omega)}
  {\delta\mu(S)}$. By (A$_3$), $f$ is $NY$-strongly fine blocked
  by $F_{n_0}$ via $\hat{g}$, which proves the theorem for
  $\epsilon\le \mu(S)$. If $\mu(S)= \mu(T)$, the proof has been
  completed. Otherwise, $\mu(T\setminus S)>0$. Let $R=T\setminus S$.
  Again by Lemma \ref{lem:lyapunov},
  \[
  G_{\mathcal Q}= {\rm cl} \left\{\left(\mu(B^{\mathcal Q}),
  \int_{B^{\mathcal Q}}(a -f) d\mu \right)\in \mathbb R\times
  Y^\Omega: B^{\mathcal Q}\in \Sigma_{R\cap T_{\mathcal Q}}\right\}
  \]
  is convex for all $\mathcal Q\in \mathfrak P(R)$. Given any
  $\alpha\in (0, 1)$ and $\mathcal Q\in \mathfrak P(R)$, applying
  an argument similar to the previous one, one can find a sequence
  $\{B_n^\mathcal Q\}\subseteq \Sigma_{R\cap T_\mathcal Q}$ such
  that $\mu(B_n^\mathcal Q)= (1-\alpha)\mu(R\cap T_Q)$ and
  for all $\omega \in \Omega$,
  \[
  \lim_{n\to \infty}\int_{B_n^\mathcal Q}(a(\cdot,\omega)-f(\cdot,
  \omega))d\mu= (1-\alpha)\kappa^\mathcal Q(\omega),
  \]
  where
  $\kappa^\mathcal Q(\omega)= \int_{R\cap T_\mathcal Q} (a(\cdot,
  \omega)- f(\cdot, \omega))d\mu$. Let
  \[
  B_n= \left(\bigcup_{\mathcal Q\in \mathfrak P(R)} B_n^\mathcal Q
  \right)\bigcup \left(\bigcup_{\mathcal Q\in \mathfrak P_R\setminus
  \mathfrak P(R)}(R\cap T_\mathcal Q)\right)
  \]
  for all $n\in \mathbb N$ and $\kappa(\omega)= \int_R (a(\cdot,\omega)
  -f(\cdot,\omega))d\mu$ for all $\omega \in \Omega$. For all $n\geq
  1$, define a function $b_n: \Omega\to Y_+$ such that
  \begin{eqnarray} \label{eqn:bn}
  b_n(\omega)=(1-\alpha)\kappa(\omega)-\int_{B_n}(a(\cdot,\omega)-
  f(\cdot,\omega))d\mu.
  \end{eqnarray}
  Then $b_n$ is $\bigvee \mathfrak P_{B_n}$-measurable for all $n\geq
  1$, and $\|b_n(\omega)\|\rightarrow 0$ as $n\to \infty$ for all
  $\omega\in \Omega$. Choose an $n_1$ satisfying $\alpha z(\omega)-
  b_{n_1}(\omega)\gg 0$ for all $\omega\in \Omega$, define
  $g_{\alpha}:S\times \Omega\to Y_+$ such that
  \[
  g_{\alpha}(t, \omega)= \alpha g(t, \omega)+ (1- \alpha) f(t,
  \omega)+ \frac{1}{\mu(S)}(\alpha z(\omega)- b_{n_1}(\omega)),
  \]
  and take $\widetilde{S}=S\cup B_{n_1}.$ Note that
  $\mu(\widetilde{S})= \mu(S)+ (1- \alpha)\mu(T\setminus S)$ and
  $g_{\alpha}$ is $\bigvee \mathfrak P_{\widetilde S}$-measurable
  for almost all $t\in S$. By (A$_3^\prime$) and (A$_4^\prime$),
  $V_{t}(g_{\alpha}(t, \cdot))> V_{t}(f(t, \cdot))$ for almost
  all $t\in S$. It remains to verify that $f$ is $NY$-strongly
  fine blocked by $\widetilde{S}$. To this end, define
  $y_{\alpha}: \widetilde{S}\times \Omega\to Y_+$ by
  \[
  y_\alpha(t, \omega) = \left\{
  \begin{array}{ll}
  g_\alpha(t, \omega), & \mbox{if $(t, \omega)\in
  S\times \Omega$;}\\[0.5em]
  f(t, \omega), & \mbox{if $(t, \omega)\in B_{n_1}\times\Omega$.}
  \end{array}
  \right.
  \]
  Then $y_{\alpha}(t, \cdot)$ is $\bigvee \mathfrak
  P_{\widetilde{S}}$-measurable and $V_{t}(y_{\alpha}(t, \cdot))
  \geq V_{t}(f(t, \cdot))$ for almost all $t\in \widetilde{S}$, and
  $V_{t} (y_{\alpha}(t, \cdot))> V_{t}(f(t, \cdot))$ for almost
  all $t\in S$. Using (\ref{eqn:zomega}) and (\ref{eqn:bn}), one
  has
  \[
  \int_{\widetilde{S}}(a(\cdot,\omega)-y_{\alpha}(\cdot,\omega))
  d\mu=(1- \alpha)\int_T (a(\cdot,\omega)-f(\cdot,\omega))
  d\mu= 0
  \]
  for all $\omega\in \Omega$. This completes the proof.
  \end{proof}

  An allocation $f$ in $\mathcal E$ is \emph{NY-fine blocked by a
  coalition $S$} \cite{Yannelis:91} if there is an $S$-exactly
  feasible assignment $g$ such that $g(t,\cdot)$ is $\bigvee
  \mathfrak P_S$-measurable and $V_t(g(t,\cdot))> V_t(f(t,\cdot))$
  for almost all $t\in S$. The \emph{NY-fine core} \cite{Yannelis:91}
  of $\mathcal E$ is the set of exactly feasible allocations which
  are not $NY$-fine blocked by any coalition of $\mathcal E$.

  \begin{remark}\label{rem:cominfeco}
  Under (A$_1$)-(A$_3$), (A$_3^\prime$)-(A$_4^\prime$) and (A$_5$),
  an analogous result can be derived for allocations not in the
  $NY$-fine core of $\mathcal E$ by modifying the functions
  $g_\alpha$ and $y_\epsilon$ in the following way:
  \[
  g_{\alpha}(t, \omega)= \alpha g(t, \omega)+ (1- \alpha) f(t,
  \omega)+ \frac{1}{\mu(S)}(\alpha z(\omega)- b_{n_1}(\omega)- x),
  \]
  and
  \[
  y_\alpha(t, \omega) = \left\{
  \begin{array}{ll}
  g_\alpha(t, \omega), & \mbox{if $(t, \omega)\in
  S\times \Omega$;}\\[0.5em]
  f(t, \omega)+ \frac{x}{\mu(B_{n_1})}, & \mbox{if $(t, \omega)
  \in B_{n_1}\times \Omega$,}
  \end{array}
  \right.
  \]
  where $x\gg 0$ such that $\alpha z(\omega)- b_{n_1}(\omega)-
  x\gg 0$.
  \end{remark}

  \begin{definition}
  An allocation $f$ in $\mathcal E$ is \emph{NY-privately blocked
  by a coalition $S$} \cite{Yannelis:91} if there exists an
  $S$-exactly feasible allocation $g$ such that $V_t(g(t,\cdot))
  > V_t(f(t,\cdot))$ for almost all $t\in S$. The \emph{NY-private
  core} \cite{Yannelis:91} of $\mathcal E$ is the set of exactly
  feasible allocations which are not $NY$-privately blocked by
  any coalition of $\mathcal E$.
  \end{definition}

  Now, we are ready to present one of the main results of this
  paper, which completely answers a question of Pesce in
  \cite[Remark 1]{Pesce:10}.

  \begin{theorem}\label{thm:privatedomonated}
  Assume that $f$ is an exactly feasible allocation in $\mathcal E$
  which is not in the $NY$-private core and $0< \epsilon< \mu(T)$.
  Under \emph{(A$_1$)-(A$_3$)}, \emph{(A$_4^\prime$)} and
  \emph{(A$_5$)}, $f$ is $NY$-privately blocked by some coalition
  $S$ with $\mu(S)= \epsilon$.
  \end{theorem}

  \begin{proof}
  Since $f$ is not in the $NY$-private core of $\mathcal E$, there
  exist a coalition $S$ and an $S$-exactly feasible allocation $g$
  such that $V_t(g(t,\cdot))> V_t(f(t,\cdot))$ for almost all
  $t\in S$. For all $\omega \in \Omega$ and $\mathcal Q\in
  \mathfrak P(S)$, let
  \[
  e_\mathcal Q(\omega) = \frac{1}{\mu(S\cap T_\mathcal Q)}\int_{S
  \cap T_\mathcal Q}a(\cdot, \omega)d\mu.
  \]
  Choose an $e\gg 0$ such that $e \le \frac{e_\mathcal Q(\omega)}
  {3}$ for all $\omega\in \Omega$ and $\mathcal Q\in \mathfrak
  P(S)$, an open ball $U$ with center $0$ and radius $\epsilon> 0$
  such that $e- U\subseteq {\rm int}Y_+$ and a  $\lambda\in (0,1)$.
  Let $\{c_m\}$ be a sequence in $(0,1)$ such that $c_m\to 0$ as
  $m\to \infty$. Pick an arbitrary element $\mathcal Q \in
  \mathfrak P(S)$, and define a function $g_m^\mathcal Q:
  (S\cap T_\mathcal Q)\times \Omega\to Y_+$ such that $g_m^\mathcal
  Q(t, \omega)= (1-c_m)g(t, \omega)+ c_m(e_\mathcal Q(\omega)- 2e)$.
  By an argument similar to that in Lemma \ref{lem: block 1}, one
  can find an increasing sequence $\{S_m^\mathcal Q\}\subseteq
  \Sigma_{S \cap T_\mathcal Q}$ such that $\bigcup_m S_m^\mathcal
  Q\sim S\cap T_\mathcal Q$, $\lim_{n\to \infty} ((S\cap T_\mathcal Q)
  \setminus S_m^\mathcal Q)= 0$ and $V_t(g_m^\mathcal Q(t, \cdot))>
  V_t(f(t, \cdot))$ for almost all $t\in S_m^\mathcal Q$. By absolute
  continuity of the Bochner integral, there is some $\delta> 0$
  such that
  \[
  \frac{2}{\mu(S\cap T_\mathcal Q)}\int_{R_\mathcal Q}\left(g(\cdot,
  \omega) - e_\mathcal Q(\omega)\right)d\mu \in U
  \]
  for all $R_\mathcal Q\in \Sigma_{S\cap T_\mathcal Q}$ with
  $\mu(R_\mathcal Q)< \delta$ and $\mathcal Q\in \mathfrak P(S)$.
  For each $\mathcal Q\in \mathfrak P(S)$, choose an $m_\mathcal Q$
  such that
  \[
  \mu\left(S_{m_\mathcal Q}^\mathcal Q\right)> \left(1-\frac{\lambda}
  {2}\right)\mu(S\cap T_\mathcal Q)
  \]
  and $\mu((S\cap T_\mathcal Q)\setminus S_{m_\mathcal Q}^\mathcal Q)<
  \delta.$ Let $m_0= \max\{m_\mathcal Q: \mathcal Q\in \mathfrak P(S)\}.$
  It follows that
  \[
  \frac{1}{\mu(S_{m_0}^\mathcal Q)}\int_{(S\cap T_\mathcal Q)\setminus
  S_{m_0}^\mathcal Q}\left(g(\cdot,\omega)- e_\mathcal Q(\omega)\right)
  d\mu\in U
  \]
  for all $\mathcal Q\in \mathfrak P(S)$.
  For each $\mathcal Q\in \mathfrak P(S)$ and $(t, \omega)\in
  S_{m_0}^\mathcal Q\times \Omega$, set
  \[
  x(t, \omega)= e_\mathcal Q(\omega)- \frac{1}{\mu(S_{m_0}^\mathcal
  Q)}\int_{(S\cap T_\mathcal Q)\setminus S_{m_0}^\mathcal Q}
  \left(g(\cdot, \omega)- e_\mathcal Q(\omega)\right)d\mu.
  \]
  Consider a function $y^{\mathcal Q}:(S\cap T_\mathcal Q)\times
  \Omega\to Y_+$ defined by
  \[
  y^{\mathcal Q}(t, \omega) = \left\{
  \begin{array}{ll}
  (1- c_{m_0}) g(t, \omega)+ c_{m_0} x(t, \omega),
  & \mbox{if $(t, \omega)\in S_{m_0}^\mathcal Q\times \Omega$;}
  \\[0.5em]
  g(t, \omega), & \mbox{otherwise.}
  \end{array}
  \right.
  \]
  Since $y^{\mathcal Q}(t, \omega)\gg g_{m_0}^\mathcal Q(t, \omega)+
  c_{m_0}e$ for all $(t, \omega)\in S_{m_0}^{\mathcal Q}\times
  \Omega$, by (A$_3$), $V_t(y^\mathcal Q(t,\cdot))> V_t(f(t,
  \cdot))$ for almost all $t\in S\cap T_\mathcal Q$ and
  $y^\mathcal Q$ is an $(S \cap T_\mathcal Q)$-allocation.
  Moreover,
  \begin{eqnarray} \label{eqn:integralyq}
  \int_{S\cap T_\mathcal Q} y^\mathcal Q(\cdot, \omega)d\mu
  = \int_{S\cap T_\mathcal Q}\left((1- c_{m_0}) g(\cdot, \omega)+
  c_{m_0} a(\cdot, \omega)\right)d\mu
  \end{eqnarray}
  for all $\omega\in \Omega$. By Lemma \ref{lem:lyapunov}, the set
  \[
  H_{\mathcal Q} = {\rm cl} \left\{\left(\mu(E^{\mathcal Q}),
  \int_{E^{\mathcal Q}}\left(y^{\mathcal Q}-a\right)
  d\mu \right)\in \mathbb R\times Y^\Omega:
  E^{\mathcal Q}\in \Sigma_{S\cap T_{\mathcal Q}}\right\}
  \]
  is convex. Using an argument similar to that in the proof
  of Theorem \ref{thm:finedomonated}, one can find a sequence
  $\{F_n^\mathcal Q\}\subseteq \Sigma_{S\cap T_{\mathcal Q}}$ such
  that $\mu(F_n^\mathcal Q)= \lambda\mu(S\cap T_{\mathcal Q})$ and
  for all $\omega \in \Omega$,
  \[
  \lim_{n\to \infty}\int_{F_n^\mathcal Q} (y^\mathcal Q(\cdot,\omega)-
  a(\cdot, \omega))d\mu = \lambda z^\mathcal Q(\omega),
  \]
  where
  \begin{eqnarray} \label{eqn:zq}
  z^\mathcal Q(\omega) = \int_{S\cap T_{\mathcal Q}} \left(y^\mathcal
  Q(\cdot,\omega)- a(\cdot,\omega)\right)d\mu.
  \end{eqnarray}
  The function $b_n^\mathcal Q: \Omega\to Y_+$, defined by
  \[
  b_n^\mathcal Q(\omega)= \lambda z^\mathcal Q(\omega)-
  \int_{F_n^\mathcal Q}\left(y^\mathcal Q(\cdot, \omega)-
  a(\cdot, \omega)\right) d\mu,
  \]
  is $\mathcal Q$-measurable for all $n\geq 1$ and $\|b_n^\mathcal
  Q(\omega)\|\to 0$ as $n\to \infty$ for all $\omega\in \Omega$. Note
  that $\min\left\{\mu(F_n^\mathcal Q\cap S_{m_0}^\mathcal Q):
  n\geq 1\right\}\geq \frac{\lambda}{2} \mu(S\cap T_{\mathcal Q})> 0$.
  Choose an $n_\mathcal Q$ such that $\frac{2 b_{n_\mathcal
  Q}^\mathcal Q(\omega)}{\lambda\mu(S\cap T_{\mathcal Q})} \in
  c_{m_0}U$ for all $\omega\in \Omega$. Then
  $c_{m_0} e+ \frac{b_{n_\mathcal Q}^\mathcal Q(\omega)}
  {\mu\left(F_{n_\mathcal Q}^\mathcal Q\cap S_{m_0}^\mathcal Q
  \right)} \gg 0$ for all $\omega\in \Omega$. Define a function
  $g^{\mathcal Q}: F_{n_\mathcal Q}^\mathcal Q\times \Omega\to
  Y_+$ such that
  \[
  g^{\mathcal Q}(t, \omega) = \left\{
  \begin{array}{ll}
  y^\mathcal Q(t, \omega)+ \frac{b_{n_\mathcal Q}^\mathcal Q(\omega)}
  {\mu\left(F_{n_\mathcal Q}^\mathcal
  Q\cap S_{m_0}^\mathcal Q\right)},
  & \mbox{if $(t, \omega)\in \left(F_{n_\mathcal Q}^\mathcal Q\cap
  S_{m_0}^\mathcal Q\right)\times \Omega$;}\\[0.5em]
  y^\mathcal Q(t, \omega), & \mbox{otherwise.}
  \end{array}
  \right.
  \]
  By (A$_3$) and the fact that $V_t\left(g_{m_0}^\mathcal Q(t,
  \cdot))> V_t(f(t, \cdot)\right)$ for almost all $t\in
  F_{n_\mathcal Q}^\mathcal Q\cap S_{m_0}^\mathcal Q$, we have
  $V_t(g^\mathcal Q(t, \cdot))> V_t(f(t, \cdot))$ for almost all
  $t\in F_{n_\mathcal Q}^\mathcal Q\cap S_{m_0}^\mathcal Q$. So,
  $g^\mathcal Q$ is $F_{n_\mathcal Q}^\mathcal Q$-allocation and
  $V_t(g^\mathcal Q(t, \cdot))> V_t(f(t, \cdot))$ for almost all
  $t\in F_{n_\mathcal Q}^\mathcal Q$. Furthermore,
  \begin{eqnarray} \label{eqn:lamdazq}
  \int_{F_{n_\mathcal Q}^\mathcal Q} (g^{\mathcal Q}(\cdot, \omega)-
  a(\cdot, \omega))d\mu = \lambda z^\mathcal Q(\omega)
  \end{eqnarray}
  for all $\omega\in \Omega$. Let $F= \bigcup\{ F_{n_\mathcal
  Q}^\mathcal Q: \mathcal Q\in \mathfrak P(S)\}$. So
  $\mu(F)= \lambda\mu(S)$. Define a function $h:F\times \Omega\to
  Y_+$ such that $h(t, \omega)= g^{\mathcal Q}(t, \omega)$ if
  $(t, \omega)\in F_{n_\mathcal Q}^\mathcal Q\times \Omega$. Then
  $h$ is an $F$-allocation and $V_t(h(t, \cdot))> V_t(f(t, \cdot))$
  for almost all $t\in F$. By (\ref{eqn:integralyq})-(\ref{eqn:lamdazq}),
  we have $\int_F (h(\cdot, \omega)- a(\cdot, \omega))d\mu = 0$
  for all $\omega\in \Omega$. Thus, $f$ is $NY$-privately blocked by
  $F$ via $h$. This proves the theorem for $\epsilon\leq
  \mu(S)$. If $\mu(S)= \mu(T)$, the proof has been completed. Otherwise,
  $\mu(T\setminus S)>0$. Let $S^\prime= \bigcup\{S_{m_0}^\mathcal Q:
  \mathcal Q\in \mathfrak P(S)\}$. Let $A= T\setminus S$ and
  $u= \frac{\lambda c_{m_0} \mu(S^\prime)e}{2(1- \lambda)\mu(A)}$.
  Again pick an arbitrary element $\mathcal Q\in \mathfrak P(A)$.
  By Lemma \ref{lem:lyapunov},
  \[
  G_{\mathcal Q} = {\rm cl} \left\{\left(\mu(B^{\mathcal Q}),
  \int_{B^{\mathcal Q}}\left(a-f- u\right) d\mu \right) \in
  \mathbb R\times Y^\Omega: B^{\mathcal Q}\in \Sigma_{A\cap
  T_\mathcal Q}\right\}
  \]
  is convex. Hence, there exists a sequence $\{B_k^\mathcal
  Q\}\subseteq \Sigma_{A\cap T_Q}$ such that $\mu(B_k^\mathcal Q)=
  (1-\lambda)\mu(A\cap T_\mathcal Q)$ and for all $\omega \in \Omega$,
  \[
  \lim_{k\to \infty}
  \int_{B_k^\mathcal Q} (a(\cdot, \omega)-f(\cdot,\omega)- u)d\mu=
  (1-\lambda)v^\mathcal Q(\omega),
  \]
  where
  \begin{eqnarray}\label{eqn:vQ}
  v^\mathcal Q(\omega) = \int_{A\cap T_\mathcal Q} (a(\cdot,\omega)-
  f(\cdot,\omega)- u)d\mu.
  \end{eqnarray}
  The function $d_k^\mathcal Q: \Omega\to Y_+$, defined by
  \[
  d_k^\mathcal Q(\omega) = (1- \lambda) v^\mathcal Q(\omega)-
  \int_{B_k^\mathcal Q}\left(a(\cdot,\omega)- f(\cdot,\omega)- u
  \right) d\mu,
  \]
  is $\mathcal Q$-measurable for all $k\geq 1$ and $\|d_k^\mathcal
  Q(\omega)\|\to 0$ as $k\to \infty$ for all $\omega\in \Omega$.
  Choose a $k_\mathcal Q$ such that
  $u- \frac{d_{k_\mathcal Q}^\mathcal Q (\omega)}{(1- \lambda)
  \mu(A\cap T_\mathcal Q)}\gg 0$
  for each $\omega\in \Omega$. It is obvious that the function
  $f^\mathcal Q: B_{k_\mathcal Q}^\mathcal Q\times \Omega\to Y_+$,
  defined by
  \[
  f^\mathcal Q (t, \omega) = f(t, \omega)+ u-
  \frac{d_{k_\mathcal Q}^\mathcal Q (\omega)}{(1- \lambda)
  \mu(A\cap T_\mathcal Q)},
  \]
  is an $B_{k_\mathcal Q}^\mathcal Q$-allocation. By (A$_3$),
  $V_t(f^\mathcal Q (t,\cdot))> V_t(f(t, \cdot))$ for almost
  all $t\in B_{k_\mathcal Q}^\mathcal Q$. Furthermore, for each
  $\omega \in\Omega$,
  \begin{eqnarray}\label{eqn:integrgalBk}
  \int_{B_{k_\mathcal Q}^\mathcal Q}(a(\cdot, \omega)- f^\mathcal
  Q(\cdot, \omega))d\mu= (1- \lambda)v^\mathcal Q(\omega).
  \end{eqnarray}
  Let $B= \bigcup \{B_{k_\mathcal Q}^\mathcal Q: \mathcal
  Q\in \mathfrak P(A)\}$. Then, $\mu(B)= (1- \lambda)\mu(A)$.
  Now, define a function $f_\lambda: B\times \Omega\to Y_+$
  such that $f_\lambda(t, \omega)= f^\mathcal Q(t, \omega)$
  if $(t,\omega)\in B_{k_\mathcal Q}^\mathcal Q\times
  \Omega$, and for any $\mathcal Q\in \mathfrak P(S)$, consider
  the function $\hat y^{\mathcal Q}:S\cap T_\mathcal Q\to Y_+$
  defined by
  \[
  \hat y^\mathcal Q(t, \omega) = \left\{
  \begin{array}{ll}
  y^\mathcal Q(t, \omega)- \frac{c_{m_0}}{2}e,
  & \mbox{if $(t, \omega)\in
  S_{m_0}^\mathcal Q\times \Omega$;}\\[0.5em]
  y^\mathcal Q(t, \omega), & \mbox{otherwise.}
  \end{array}
  \right.
  \]
  Since $\hat y^{\mathcal Q}(t, \omega)\gg g_{m_0}^\mathcal Q(t,
  \omega)+ \frac{c_{m_0}}{2}e$ for all $(t, \omega)\in
  S_{m_0}^{\mathcal Q} \times \Omega$, by (A$_3$), $V_t(\hat
  y^\mathcal Q(t, \cdot))> V_t(f(t, \cdot))$ for almost all
  $t\in S \cap T_{\mathcal Q}$. Note that $\hat y^\mathcal Q$ is
  an $(S\cap T_\mathcal Q)$-allocation. Take $\widehat S=
  \bigcup \{S\cap T_\mathcal Q: \mathcal Q\in \mathfrak P(S)\}$.
  Then, $\mu(\widehat S)= \mu(S)$. Define
  $y_\lambda: \widehat S\times \Omega\to Y_+$ by $y_\lambda (
  t, \omega)= \hat y^\mathcal Q(t, \omega)$ if $(t, \omega)\in
  (S\cap T_\mathcal Q)\times \Omega$. It can be checked that
  for each $\omega\in \Omega$,
  \begin{eqnarray}\label{eqn:intwidehatS}
  \int_{\widehat S} a(\cdot, \omega)d\mu - \int_{\widehat S}
  y_\lambda(\cdot, \omega)d\mu= \frac{c_{m_0}\mu(S^\prime)}{2}e.
  \end{eqnarray}
  Consider $h_\lambda:\widehat S\times \Omega\to Y_+$ defined by
  $h_\lambda(t, \omega)= \lambda y_\lambda(t, \omega)+(1- \lambda)
  f(t, \omega)$. By (A$_4^\prime$), $V_{t}(h_\lambda(t, \cdot))> V_{t}
  (f(t, \cdot))$ for almost all $t\in \widehat{S}$, and further
  $h_\lambda$ is an $\widehat S$-allocation. Let $\widetilde{S}=
  \widehat S\cup B$. Since $\mu(\widetilde{S})= \mu(S)+ (1-
  \lambda)\mu(T\setminus S)$, it remains to verify that
  $f$ is $NY$-privately blocked by $\widetilde{S}$. To show
  this, consider $g_\lambda: \widetilde{S}\times \Omega\to Y_+$
  defined by
  \[
  g_\lambda(t, \omega) = \left\{
  \begin{array}{ll}
  h_\lambda(t, \omega), & \mbox{if $(t, \omega)\in
  \widehat S\times \Omega$;}\\[0.5em]
  f_\lambda(t, \omega), & \mbox{if $(t, \omega)\in B\times \Omega$.}
  \end{array}
  \right.
  \]
  Obviously, $g_\lambda$ is an $\widetilde{S}$-allocation and $V_{t}
  (g_\lambda(t, \cdot))> V_{t}(f(t, \cdot))$ for almost all
  $t\in \widetilde{S}$. Furthermore, using (\ref{eqn:vQ})-
  (\ref{eqn:intwidehatS}), it can be simply verified that
  \[
  \int_{\widetilde{S}}(a(\cdot, \omega)-g_\lambda (\cdot, \omega))
  d\mu= (1- \lambda)\int_{T}(a(\cdot, \omega)-f(\cdot, \omega))
  d\mu= 0
  \]
  holds for all $\omega\in \Omega$. This completes the proof.
  \end{proof}

  \begin{remark}
  If $Y$ is separable, then without (A$_4^\prime$) the conclusions
  of Theorem \ref{thm:finedomonated}, Remark \ref{rem:cominfeco}
  and Theorem \ref{thm:privatedomonated} hold. Indeed, to restore
  the conclusions in Theorem \ref{thm:finedomonated} and Remark
  \ref{rem:cominfeco}, note that $\int_S g d\mu, \int_S f d\mu$
  are in the convex set ${\rm cl}\int_S P_f d\mu$. So, $\int_S
  (\alpha g+ (1-\alpha)f)d\mu\in {\rm cl}\int_S P_f d\mu$ and
  by (A$_3^\prime$), $\int g_\alpha d\mu \in \int_S P_f d\mu$.
  Similarly, to restore the conclusion of Theorem
  \ref{thm:privatedomonated}, note that $\int_S g_{m_0}^\mathcal Q
  d\mu$ and $\int_S f d\mu$ are elements of the convex set
  ${\rm cl}\int_S P_f d\mu$. Thus, $\int_S (\lambda
  g_{m_0}^\mathcal Q+ (1-\lambda)f) d\mu\in {\rm cl}\int_S
  P_f d\mu$ and by (A$_3$), $\int h_\lambda d\mu\in \int_S P_f d\mu$.
  \end{remark}

  \section{Robust efficiency and different types of\\
  cores of mixed market economies}\label{sec:mixed}

  In this section, we study cores and Walrasian expectations
  allocations in mixed economies. We characterize Walrasian
  expectations allocations in terms of robust efficiency,
  and establish relationships among various types of cores.
  To achieve these goals, we associate the mixed economy
  $\mathcal{E}$ in Section \ref{sec:model} with an atomless
  economy $\mathcal{E}^\ast$, and then apply results
  established in Section \ref{sec:atomless}.
  The space of agents of $\mathcal{E}^\ast$ is denoted
  by $(T^\ast, \Sigma^\ast, \mu^\ast)$, where $T^\ast = T_0\cup
  T_1^\ast$ and $T_1^\ast$ is an atomless measure space such that
  $\mu^\ast(T_1^\ast)= \mu(T_1)$ and $T_0\cap T_1^\ast= \emptyset$.
  We assume that $(T^\ast, \Sigma^\ast, \mu^\ast)$ is obtained by
  the direct sum of $(T_0, \Sigma_{T_0}, \mu_{T_0})$ and the
  measure space $T_1^\ast$, where $\mu_{T_0}$ is the restriction
  of $\mu$ to $T_0$. It is also assumed that each agent $A\in
  T_1$ one-to-one corresponds to a measurable subset $A^\ast$
  of $T_1^\ast$ with $\mu^\ast(A^\ast)= \mu(A)$. Each agent
  $t\in A^\ast$ is characterized by the private information
  set $\mathcal{F}_t = \mathcal{F}_{A}$; the consumption set
  $Y_+$ in each state $\omega \in \Omega$; the initial endowment
  $a(t,\cdot) =a(A, \cdot)$; the utility function $U_t = U_{A}$; and
  the prior $q_t = q_{A}$. Therefore, the ex ante expected utility
  function of every agent $t\in A^\ast$ is $V_t=V_{A}$.

  \subsection{Robust efficiency} \label{sec:eqthm}
  In this subsection, we characterize a Walrasian expectations
  equilibrium of a mixed economy by the private blocking power of
  the grand coalition. For any coalition $S$, allocation $f$ in
  $\mathcal E$ and any $0\leq r\leq 1$, we introduce an
  asymmetric information economy $\mathcal E(S, f, r)$ which
  coincides with $\mathcal E$ except for the initial endowment
  allocation that is given by
  \[
  a(S, f, r)(t, \cdot) = \left\{
  \begin{array}{ll}
  a(t, \cdot), & \mbox{if $t\in T\setminus S$;}\\[0.5em]
  (1- r)a(t, \cdot)+ r f(t, \cdot), & \mbox{if $t\in S$.}
  \end{array}
  \right.
  \]
  A feasible allocation $f$ in $\mathcal E$ is said to be
  \emph{robustly efficient} \cite{Herves-Beloso-Moreno-Garcia:08}
  if $f$ is not privately blocked by the grand coalition in
  every economy $\mathcal E(S, f, r)$.

  \begin{lemma}\label{lem:privateblocked}
  Assume that an allocation $f^\ast$ in $\mathcal E^\ast$ is
  privately blocked by a coalition $S^\ast$ with
  $\mu^\ast(S^\ast\cap T_1^\ast) >0$. Under \emph{(A$_1$)-(A$_2$)}
  and \emph{(A$_5$)}, for any $0 <\epsilon \le
  \mu^\ast(S^\ast\cap T_1^\ast)$, there exist a coalition $R^\ast
  \subseteq \bigcup_{\mathcal Q\in \mathfrak P(T^\ast)}(S^\ast\cap
  T_\mathcal Q^\ast)$, a sub-coalition $R_1^\ast$ of $R^\ast$ and
  an $R^\ast$-allocation $g^\ast$ such that
  \begin{enumerate}
  \item[(i)] $\int_{R^\ast} (a(\cdot, \omega)- g^\ast(\cdot,
         \omega))d\mu^\ast\gg 0$ for all $\omega\in \Omega$
         and $V_t(g^\ast(t, \cdot))> V_t(f^\ast(t, \cdot))$ for
         almost all $t\in R^\ast$,

  \item[(ii)] $g^\ast(t, \omega)\gg 0$ for all $(t, \omega)\in
  R_1^\ast\times \Omega$ and $\mathfrak P(R_1^\ast)=
      \mathfrak P(R^\ast)$,

  \item[(iii)] $\mu^\ast(R^\ast\cap T_1^\ast)= \epsilon$ and
  $\mu^\ast(R^\ast\cap T_\mathcal Q^\ast)= \frac{\epsilon
  \mu^\ast(S^\ast\cap T_\mathcal Q^\ast)}{\mu^\ast(S^\ast\cap
  T_1^\ast)}$ for all $\mathcal Q\in \mathfrak P(S^\ast)$.
  \end{enumerate}
  \end{lemma}

  \begin{proof}
  If $\epsilon = \mu^\ast(S^\ast\cap T_1^\ast)$, the conclusion
  directly follows from Lemma \ref{lem: block 1}. Assume
  $0<\epsilon < \mu^\ast(S^\ast\cap T_1^\ast)$. Let
  $\delta = \frac{\epsilon}{\mu^\ast(S^\ast\cap T_1^\ast)}$ and
  $\alpha= 1- \frac{\delta}{2}$. Applying Lemma \ref{lem: block 1},
  one has a sub-coalition $S_1^\ast$ of $S^\ast$ and an
  $S^\ast$-allocation $g^\ast$ satisfying (i)-(iii) of Lemma
  \ref{lem: block 1}. For each $\mathcal Q\in \mathfrak P(S^\ast)$,
  by Lemma \ref{lem:lyapunov}, the set
  \[
  H_{\mathcal Q}= {\rm cl}\left\{\left(\mu^\ast(E^\mathcal Q),
  \mu^\ast(E^\mathcal Q\cap T_1^\ast), \int_{E^\mathcal Q}(a-
  g^\ast)d\mu^\ast\right) \in {\mathbb R}^2 \times Y^\Omega:
  E^\mathcal Q\in \Sigma^\ast_{S^\ast\cap T_{\mathcal Q}^\ast}
  \right\}
  \]
  is convex. Similar to the proof of Theorem
  \ref{thm:finedomonated}, for each $\mathcal Q\in \mathfrak
  P(S^\ast)$, there exists a sequence $\{E_n^\mathcal Q\}\subseteq
  \Sigma^\ast_{S^\ast\cap T_{\mathcal Q}^\ast}$ such that
  $\mu^\ast(E_n^\mathcal Q)= \delta \mu^\ast(S^\ast\cap T_{\mathcal
  Q}^\ast)$, $\mu^\ast(E_n^\mathcal Q\cap T_1^\ast)= \delta
  \mu^\ast(S^\ast\cap T_{\mathcal Q}^\ast\cap T_1^\ast)$ and
  \[
  \lim_{n\to \infty} \int_{E_n^\mathcal Q}(a- g^\ast)d\mu^\ast=
  \delta \int_{S^\ast\cap T_{\mathcal Q}^\ast}(a- g^\ast)d
  \mu^\ast.
  \]
  Since $\mu^\ast(S_1^\ast\cap T_\mathcal Q^\ast)> \alpha
  \mu^\ast(S^\ast\cap T_\mathcal Q^\ast)$ for all $\mathcal Q
  \in \mathfrak P(S^\ast)$, then $\mu^\ast(S_1^\ast\cap
  E_n^\mathcal Q)> 0$ for all $n\ge 1$ and all $\mathcal Q\in
  \mathfrak P(S^\ast)$. Let $E_n= \bigcup_{\mathcal Q\in
  \mathfrak P(S^\ast)} E_n^\mathcal Q$ for all $n\geq 1$.
  Then
  \[
  \lim_{n\to \infty} \int_{E_n}(a- g^\ast)d\mu^\ast=
  \delta \int_{S^\ast}(a- g^\ast)d\mu^\ast.
  \]
  Pick an $n_0$ such that $ \int_{E_{n_0}} (a- g^\ast) d
  \mu^\ast\gg 0$, and put $R^\ast= E_{n_0}$, $R_1^\ast=
  R^\ast\cap S_1^\ast$.
  \end{proof}

  \begin{lemma}\cite{Evern-Husseinov:08} \label{lem:corwal}
  Assume $Y$ is separable. Under \emph{(A$_1$)-(A$_3$)} and
  \emph{(A$_5$)}, $f^\ast$ is a Walrasian expectations allocation
  of $\mathcal E^\ast$ if and only if it is in the private
  core of $\mathcal E^\ast$.
  \end{lemma}

  \begin{lemma} \cite{Evern-Husseinov:08} \label{lem:freeblock}
  Assume that $\mathcal E$ satisfies \emph{(A$_1$)-(A$_3$)} and
  \emph{(A$_5$)}. Let $f^\ast$ be a feasible allocation of
  $\mathcal E^\ast$ and $0< \epsilon< \mu^\ast(T^\ast)$.
  If $f^\ast$ is not in the private core of $\mathcal E^\ast$,
  then there is a coalition $S$ with $\mu^\ast(S)= \epsilon$
  privately blocking $f^\ast$.
  \end{lemma}

  The following lemma is similar to Theorem 3.5 in
  \cite{De Simone-Graziano:03}.

  \begin{lemma}\label{lem:auxrob}
  Assume that $f$ is a robustly efficient allocation of $\mathcal E$.
  Under \emph{(A$_1$)-(A$_7$)}, there is an allocation
  $\hat f$ in $\mathcal E$ such that $\hat f|_{T_0\times \Omega}= f$,
  $\hat f(\cdot, \omega)$ is constant on $T_1$ for each $\omega\in
  \Omega$, $V_{T_1}(\hat f(t, \cdot))= V_{T_1}(f(t, \cdot))$ for all
  $t\in T$ and $\int_T f d\mu= \int_T \hat f d\mu$.
  \end{lemma}

  \begin{proof}
  Consider the allocation $\hat f:T\times \Omega\to Y_+$ defined by
  \[
  \hat f(t, \omega)= \left\{
  \begin{array}{ll}
  f(t, \omega), & \mbox{if $(t, \omega)\in T_0\times \Omega$;}\\[0.5em]
  \frac{1}{\mu(T_1)}\int_{T_1} f(\cdot, \omega)d \mu, & \mbox{if $(t,
  \omega)\in T_1\times \Omega$.}
  \end{array}
  \right.
  \]
  To complete the proof, one only needs to verify $V_{T_1}(\hat f(t,
  \cdot))= V_{T_1}(f(t, \cdot))$ holds for all $t\in T_1$. Suppose
  that there exists a coalition $D\subseteq T_1$ such that
  $V_{T_1}(\hat f(t, \cdot))> V_{T_1}(f(t, \cdot))$ for all $t\in D$.
  Then applying an argument similar to that in Lemma \ref{lem: block 1},
  one can find some $r_1\in (0, 1)$ and a sub-coalition $C\subseteq D$
  such that $V_{T_1}(r_1 \hat f(t, \cdot))> V_{T_1}(f(t,\cdot))$
  for all $t\in C$. Let $r_2= \frac{\mu(C)}{\mu(T_1)}$ and
  $r_3= r_1+ \eta$ for some $\eta> 0$ such that $r_3\in (0, 1)$.
  Then $r_2\in (0, 1]$. Suppose that for each $\omega\in \Omega$,
  \[
  \alpha(\omega)= r_2 r_3\left(\int_T f(\cdot, \omega)d \mu- \int_T
  a(\cdot,\omega)d \mu\right)-r_2(1- r_3) \int_{T_1} a(\cdot,
  \omega)d \mu.
  \]
  Note that $\alpha(\omega)\in - {\rm int} Y_+$ for each
  $\omega\in \Omega$. Choose an $\epsilon> 0$ such that
  for each $\omega\in \Omega$, $\alpha(\omega)+
  B(0, 2\epsilon)\subseteq -{\rm int} Y_+$. By Lemma
  \ref{lem:lyapunov}, $H= {\rm cl}\left\{\int_E
  (f- a)\in Y^\Omega: E\in \Sigma_{T_0}\right\}$ is convex. So
  there is an $E_0\in \Sigma_{T_0}$ such that $\|\int_{E_0}
  (f- a)- r_2 r_3 \int_{T_0} (f-a)\|< \epsilon$. Pick an $u\in
  B(0, \epsilon)\cap {\rm int}Y_+$ and put $S= E_0\cup C$. Then,
  $\mu(S)< \mu(T)$. Note that the function $g: S\times \Omega
  \to Y_+$, defined by
  \[
  g(t, \omega)= \left\{
  \begin{array}{ll}
  \hat f(t, \omega)+ \frac{u}{2\mu(E_0)}, & \mbox{if $(t, \omega)
  \in E_0\times \Omega$;}\\[0.5em]
  r_3 \hat f(t, \omega)+ \frac{u}{2\mu(C)}, & \mbox{if $(t,
  \omega)\in C\times \Omega$,}
  \end{array}
  \right.
  \]
  is an $S$-allocation and $V_t(g(t, \cdot))> V_t(f(t, \cdot))$
  for almost all $t\in S$. Further, $g(t, \omega)\gg 0$ for
  all $(t, \omega)\in S\times \Omega$ and
  $\int_S g(\cdot, \omega)d\mu= \int_{E_0} f(\cdot, \omega)d\mu+
  r_2 r_3\int_{T_1} f(\cdot, \omega)d\mu+ u$
  for all $\omega\in\Omega$. By (A$_6$), $\int_C a(\cdot,\omega)
  d \mu=r_2 \int_{T_1} a(\cdot, \omega)d \mu$ for all $\omega\in
  \Omega$. Then it can be easily verified that for all $\omega
  \in \Omega$,
  \[
  -\alpha(\omega)+ \int_S (g(\cdot, \omega)- a(\cdot, \omega))d
  \mu=\int_{E_0} (f- a)- r_2 r_3 \int_{T_0} (f- a)+ u\in
  B(0, 2\epsilon).
  \]
  It follows that $\int_S g(\cdot, \omega)d\mu- \int_S a(\cdot,
  \omega) d\mu\ll 0$ for all $\omega\in \Omega$. Select an
  $z\gg 0$ such that $\int_S a(\cdot,\omega)d\mu- \int_S g(\cdot,
  \omega)d\mu\gg z$ for each $\omega\in \Omega$ and pick an $r
  \in (0, 1)$ such that $r_1 \hat f(t, \omega)\le r g(t,\omega)$
  for all $(t, \omega)\in C\times \Omega$. Note that the
  function $h_1:C\times \Omega\to Y_+$, defined by $h_1(t,\omega)
  = r_1 \hat f(t, \omega)$, is a $C$-allocation and $V_{T_1}
  (h_1(t, \cdot))> V_{T_1}(f(t, \cdot))$ for all $t\in C$.
  By Lemma \ref{lem: block 2}, there is an $E_0$-allocation
  $h_2: E_0\times \Omega\to Y_+$ such that $V_t(h_2(t,\cdot))>
  V_t(f(t, \cdot))$ for almost all $t\in E_0$, and
  \[
  \int_{E_0}
  h_2(\cdot, \omega)d \mu= \int_{E_0}\left(r  g(\cdot,\omega)+
  (1- r)f(\cdot, \omega)\right)d \mu
  \]
  for all $\omega\in \Omega$. Now, $h:S\times\Omega\to Y_+$,
  defined by
  \[
  h(t, \omega)= \left\{
  \begin{array}{ll}
  h_2(t, \omega), & \mbox{if $(t, \omega)\in E_0\times
  \Omega$;}\\[0.5em]
  h_1(t, \omega), & \mbox{if $(t, \omega)\in C\times \Omega$,}
  \end{array}
  \right.
  \]
  is an $S$-allocation, $V_t(h(t, \cdot))> V_t(f(t, \cdot))$
  for almost all $t\in S$, and
  \begin{eqnarray} \label{eqn:integralh}
  \int_{S}h(\cdot, \omega)d \mu\le \int_{S}(r g(\cdot, \omega)+
  (1- r)f(\cdot,\omega))d \mu \mbox{ for all } \omega\in
  \Omega.
  \end{eqnarray}
  Define a function $y:T\times \Omega\to Y_+$ such that
  \[
  y(t, \omega)= \left\{
  \begin{array}{ll}
  h(t, \omega), & \mbox{if $(t, \omega)\in S\times \Omega$;}
  \\[0.5em]
  f(t, \omega)+ \frac{rz}{\mu(T\setminus S)}, & \mbox{if
  $(t,\omega) \in (T\setminus S)\times \Omega$.}
  \end{array}
  \right.
  \]
  By (A$_3$), $V_t(y(t, \cdot))> V_t(f(t, \cdot))$ for almost all $t
  \in T\setminus S$. Thus, $y$ is an allocation and $V_t(y(t, \cdot))
  > V_t(f(t, \cdot))$ for almost all $t\in T$. Furthermore, using
  (\ref{eqn:integralh}) and $\int_{S}(a(\cdot, \omega)- g(\cdot,
  \omega))d \mu\gg z$, one can simply verify that for each
  $\omega \in \Omega$,
  \[
  \int_T (y(\cdot,\omega)- a(T\setminus S, f, r)(\cdot, \omega))d
  \mu\le (1- r)\int_{T} (f(\cdot, \omega)- a(\cdot, \omega))d \mu
  \leq 0.
  \]
  This means that $f$ is privately blocked by the grand coalition
  in $\mathcal{E}(T\setminus S, f, r)$, which contradicts with
  the fact that $f$ is robustly efficient. So $V_{T_1}(f(t,\cdot))
  \ge V_{T_1}(\hat f(t, \cdot))$ for all $t\in T_1$. Suppose that
  there is a coalition $W\subseteq T_1$ such that $V_{T_1}(f(t,
  \cdot))> V_{T_1}(\hat f(t,\cdot))$ for all $t\in W$. By (A$_4$),
  one can easily derive
  \[
  V_{T_1}
  (\hat f(t, \cdot))> \frac{1}{\mu(T_1)}\int_{T_1} V_{T_1}(\hat f(t,
  \cdot))d\mu= V_{T_1}(\hat f(t, \cdot)),
  \]
  which is a contradiction.
  Thus, $V_{T_1}(f(t, \cdot))= V_{T_1}(\hat f(t, \cdot))$ for all
  $t\in T_1$.
  \end{proof}

  Next, in answering a question mentioned Herv\'{e}s-Beloso and
  Moreno-Garc\'{i}a in \cite[p.705]{Herves-Beloso-Moreno-Garcia:08},
  we provide a characterization of Walrasian expectations equilibria
  by the veto power of the grand coalition in a mixed economy with
  asymmetric information and an ordered separable Banach space whose
  positive cone has an interior point as the commodity space.

  \begin{theorem}\label{thm:walrob}
  Assume that $Y$ is separable. Under \emph{(A$_1$)-(A$_7$)}, $f$
  is a Walrasian expectations allocation of $\mathcal E$ if and
  only if it is a robustly efficient allocation of $\mathcal E$.
  \end{theorem}

  \begin{proof}
  Suppose that $f$ is a Walrasian expectations allocation of
  $\mathcal E$. Applying an argument similar to that in
  \cite{Herves-Beloso-Moreno-Garcia:08}, one can show that
  it is robustly efficient.

  Conversely, let $f$ be a robustly efficient allocation of
  $\mathcal E$. By Lemma \ref{lem:auxrob}, there is an
  allocation $\hat f$ in $\mathcal E$ such that
  ${\hat f}|_{T_0\times \Omega} =f$, $\hat f (\cdot, \omega)$
  is a constant ${\bf c}(\omega)$ on $T_1$ for each $\omega\in
  \Omega$, $V_{T_1}(\hat f(t, \cdot))= V_{T_1}(f(t, \cdot))$
  for all $t\in T_1$ and $\int_T f d\mu= \int_T \hat f d\mu$.
  Suppose that $f$ is not a Walrasian expectations allocation of
  $\mathcal E$. Then $\hat f$ is not a Walrasian expectations
  allocation for $\mathcal E$. To see this, let $(\hat f,\pi)$
  be a Walrasian expectations equilibrium for $\mathcal E$,
  $d\in {\rm int} Y_+$ and $\alpha> 0$. By (A$_3$), one has
  $V_t(f(t,\cdot)+ \alpha d)> V_t(f(t, \cdot))= V_t(\hat f(t,
  \cdot))$ for all $t\in T$. It follows that for almost all
  $t\in T$,
  \[
  \sum_{\omega\in \Omega}\langle \pi(\omega), f(t,
  \omega)+ \alpha d\rangle> \sum_{\omega\in \Omega}\langle
  \pi(\omega), \hat f(t, \omega)\rangle.
  \]
  Letting $\alpha\to 0$,
  one has $\sum_{\omega\in \Omega}\langle \pi(\omega), f(t,
  \omega)\rangle\ge \sum_{\omega\in \Omega}\langle \pi(\omega),
  \hat f(t, \omega)\rangle$. So,
  \[
  \sum_{\omega\in \Omega}\langle \pi(\omega), f(t, \omega)
  \rangle= \sum_{\omega\in \Omega}\langle \pi(\omega),\hat
  f(t, \omega)\rangle\le \sum_{\omega\in \Omega}\langle
  \pi(\omega), a(t, \omega)\rangle
  \]
  holds for almost all $t\in T$, and one has a contradiction.
  Therefore, the allocation $\hat f^\ast:T^\ast\times \Omega
  \to Y_+$ defined by
  \[
  \hat f^\ast (t, \omega)= \left\{
  \begin{array}{ll}
  \hat{f}(t, \omega), & \mbox{if $(t, \omega)\in T_0\times
  \Omega$;} \\[0.5em]
  {\bf c}(\omega), & \mbox{if $(t,\omega) \in T_1^\ast
  \times \Omega$,}
  \end{array}
  \right.
  \]
  is not a Walrasian expectations allocation of $\mathcal
  E^\ast$. By Lemma \ref{lem:corwal}, $\hat f^\ast$
  is not in the private core of $\mathcal E^\ast$. Pick any
  $A_0\in T_1$ with $\mu(A_0)= \epsilon> 0$. According to
  Lemma \ref{lem:freeblock}, $\hat f^\ast$ is privately
  blocked by a coalition $S^\ast$ of $\mathcal E^\ast$ with
  $\mu^\ast(S^\ast)= \mu^\ast(T_0)+ \epsilon$, which yields
  $\mu^\ast(S^\ast\cap T_1^\ast)\geq \epsilon$. By Lemma
  \ref{lem:privateblocked}, there exists a coalition $R^\ast
  \subseteq S^\ast$, a sub-coalition $R_1^\ast$ of $R^\ast$
  and an $R^\ast$-allocation $g^\ast$ such that (i)-(iii) of
  Lemma \ref{lem:privateblocked} hold. Take a coalition $E$
  of $\mathcal E$ such that $E= (R^\ast\cap T_0)\cup A_0$,
  and define a function $\tilde g: E\times \Omega\to Y_+$
  by
  \[
  \tilde g(t, \omega)= \left\{
  \begin{array}{ll}
  g^\ast(t, \omega), & \mbox{if $(t, \omega)\in (R^\ast\cap
  T_0)\times \Omega$;}\\[0.5em]
  \frac{1}{\epsilon}\int_{R^\ast\cap T_1^\ast}g^\ast(\cdot,
  \omega)d \mu^\ast, & \mbox{otherwise.}
  \end{array}
  \right.
  \]
  Further, define another function $\tilde g^\ast:E^\ast
  \times \Omega\to Y_+$ such that
  \[
  \tilde g^\ast(t, \omega)= \left\{
  \begin{array}{ll}
  \tilde g(t, \omega), & \mbox{if $(t, \omega)\in (R^\ast\cap
  T_0)\times \Omega$;}\\[0.5em]
  \tilde g(A_0, \omega), & \mbox{if $(t,\omega)\in A_0^\ast
  \times\Omega$.}
  \end{array}
  \right.
  \]
  By (A$_4$), one concludes that $\tilde g^\ast$ is an
  $E^\ast$-allocation such that $V_t(\tilde g^\ast(t, \cdot))> V_t
  (\hat f^\ast(t, \cdot))$ for almost all $t\in E^\ast$ and
  $\int_{E^\ast} \tilde g^\ast(\cdot, \omega) d\mu^\ast\ll
  \int_{E^\ast}a(\cdot, \omega) d\mu^\ast$ for all $\omega \in
  \Omega$. Select some $b\gg 0$ such that $\int_{E^\ast}(a(\cdot,
  \omega) d\mu^\ast- \tilde g^\ast(\cdot, \omega)) d\mu^\ast\gg
  b$ for all $\omega\in \Omega$, and consider the function
  $g_b^\ast:E^\ast \times \Omega\to Y_+$ defined by
  $g_b^\ast (t, \omega)= \tilde g^\ast(t, \omega)+ \frac{b
  }{2 \mu^\ast(E^\ast)}$. By (A$_3$), $V_t(g_b^\ast
  (t, \cdot))> V_t(\hat f^\ast(t, \cdot))$ for almost all $t\in
  E^\ast$. Note that the function $g_b:E\times \Omega\to Y_+$,
  defined by
  \[
  g_b(t, \omega)= \left\{
  \begin{array}{ll}
  g_b^\ast(t, \omega), & \mbox{if $(t, \omega)\in (E
  \cap T_0) \times \Omega$;}\\[0.5em]
  \frac{1}{\epsilon}\int_{A_0^\ast} g_b^\ast(\cdot, \omega)d
  \mu^\ast, & \mbox{otherwise,}
  \end{array}
  \right.
  \]
  is an $E$-allocation such that $V_t(g_b(t, \cdot))>
  V_t(f(t, \cdot))$ for almost all $t\in E$. Choose an $r\in (0, 1)$
  satisfying $\tilde g(A_0, \omega)\le r g_b(A_0, \omega)$ for each
  $\omega\in \Omega$. By Lemma \ref{lem: block 2}, there exists an
  $(E\cap T_0)$-allocation $h_b$ such that $V_t(h_b(t, \cdot))>
  V_t(f(t, \cdot))$ for almost all $t\in E\cap T_0$ and
  $\int_{E\cap T_0}h_b(\cdot, \omega)d \mu= \int_{E\cap T_0}(r
  g_b(\cdot, \omega)+ (1- r) f(\cdot, \omega))d \mu$.
  Finally, consider the function $h:E\times\Omega\to Y_+$ defined by
  \[
  h(t, \omega)= \left\{
  \begin{array}{ll}
  h_b(t, \omega), & \mbox{if $(t, \omega)\in (E\cap T_0)\times
  \Omega$;}\\[0.5em]
  \tilde g(t, \omega), & \mbox{otherwise.}
  \end{array}
  \right.
  \]
  Note that $h$ is an $E$-allocation. Applying an argument similar
  to the final part of Lemma \ref{lem:auxrob}, one can show that
  $f$ is not robustly efficient. This is a contradiction and so
  $f$ is a Walrasin expectations allocation.
  \end{proof}

  \begin{remark}
  It is clear that the conclusion of Theorem \ref{thm:walrob} is
  valid in atomless economies whenever assumptions (A$_1$)-(A$_3$)
  and (A$_5$) hold. By Lemma \ref{lem:freeblock}, the coalition $S$
  of $\mathcal E(S, f, r)$ in Theorem \ref{thm:walrob} can be chosen
  arbitrarily small in an atomless economy. Thus, perturbation of
  small coalition is enough to characterize the Walrasian expectations
  allocations.
  \end{remark}

  \subsection{The $RW$-fine core and the ex-post core}
  \label{subsec:fineexpost}

  In this subsection, we establish a relationship between the
  $RW$-fine core and the ex-post core of $\mathcal E$.
  An \emph{information structure} for a coalition $S$ is a family
  $\{\mathcal{G}_t: t\in S\}$ of $\sigma$-algebras such that
  ${\mathcal G}_t \subseteq \mathcal F$ for all $t\in S$ and
  $\{t\in S: \mathcal{G}_t= \mathcal H\} \in \Sigma$ for every
  $\sigma$-algebra ${\mathcal H} \subseteq \mathcal F$. Since
  $\Omega$ is finite, the family $\{{\mathcal G}
  \subseteq {\mathcal F}: {\mathcal G} \mbox{ is a
  $\sigma$-algebra}\}$ is finite. Thus, it is possible
  that for an information structure $\{\mathcal{G}_t: t\in S\}$
  of $S$ and two distinct agents $t$ and $t'$ of $S$,
  $\mathcal G_t= \mathcal G_{t'}$.
  A \emph{communication system} for a coalition
  $S$ is an information structure $\{\mathcal{G}_t: t\in S\}$
  for $S$ such that $\mathcal F_t\subseteq \mathcal{G}_t\subseteq
  \bigvee \mathfrak P_S$ for almost all $t\in S$, and it is called
  a \emph{full communication system} if $\mathcal{G}_t= \bigvee
  \mathfrak P_S$ for almost all $t\in S$.
  Further, for any $\sigma$-algebra $\mathcal H$ with $\mathcal H
  \subseteq \mathcal F$, $\mathcal F$-measurable function $f:\Omega
  \to Y_+$ and $t\in T$, let $\mathbb E_t[f|\mathcal H]$ be the
  conditional expectation of $f$ given $\mathcal H$ with respect
  to $q_t$. For any coalition $S$, we now assume that an
  $S$-allocation (including initial endowment) is a function $f:S
  \times \Omega\to Y_+$ such that $f(\cdot, \omega) \in L_1^S(\mu,
  Y_+)$ for each $\omega\in \Omega$ and $f(t, \cdot)$ is $\mathcal
  F$-measurable for almost all $t\in S$. As mentioned previously,
  $T$-allocations are simply called \emph{allocations}.

  \begin{definition} \cite{Wilson:78}
  An allocation $f$ in $\mathcal E$ is \emph{RW-fine\footnote{RW
  is the abbreviation of Robert Wilson.} blocked} by a coalition
  $S$ if there are an $S$-allocation $g$, a communication system
  $\{\mathcal{G}_t\}_{t\in S}$ for $S$, and a nonempty event
  $A\in \bigcap_{t\in S} \mathcal G_t$ such that $\int_{S}g(\cdot,
  \omega)d\mu= \int_{S}a(\cdot, \omega) d\mu$ for all $\omega\in
  A$, and
  \[
  \mathbb E_t[U_t(\cdot, g(t, \cdot))| \mathcal{G}_t](\omega)>
  \mathbb E_t[U_t(\cdot, f(t, \cdot))| \mathcal{G}_t](\omega)
  \]
  for all $\omega\in A$ and almost all $t\in S$. The \emph{RW-fine
  core} of $\mathcal E$ is the set of all feasible allocations that
  cannot be $RW$-fine blocked by any coalition.
  \end{definition}

  \begin{definition} \cite{Einy-Moreno-Shitovitz:00}
  An allocation $f$ in $\mathcal E$ is \emph{ex-postly blocked} by
  a coalition $S$ if there exist an $S$-allocation $g$ and a state
  $\omega_0\in \Omega$ such that $\int_S g(\cdot, \omega_0)d\mu=
  \int_S a(\cdot,\omega_0)d\mu$, and $U_t(\omega_0, g(t,\omega_0))
  > U_t(\omega_0, f(t, \omega_0))$ for almost all $t\in S$. The
  \emph{ex-post core} of $\mathcal E$ is the set of all feasible
  allocations that cannot be ex-postly blocked by any coalition.
  \end{definition}

  \begin{lemma}\label{lem:fineex}
  Assume that $f$ is in the $RW$-fine core of $\mathcal E$. Under
  \emph{(A$_1$)-(A$_9$)}, there exists an allocation $\hat f$ in
  $\mathcal E$ such that $\hat f|_{T_0\times \Omega}=f$, $\hat
  f(\cdot, \omega)$ is constant on $T_1$ for each $\omega\in \Omega$,
  $U_{T_1}(\omega, \hat f(t, \omega))= U_{T_1}(\omega, f(t, \omega))$
  for all $(t,\omega)\in T_1 \times \Omega$ and $\int_T f d\mu=
  \int_T \hat f d\mu$.
  \end{lemma}

  \begin{proof}
  Consider the allocation $\hat f:T\times \Omega\to Y_+$ defined by
  \[
  \hat f(t, \omega)= \left\{
  \begin{array}{ll}
  f(t, \omega), & \mbox{if $(t, \omega)\in T_0\times \Omega$;}\\[0.5em]
  \frac{1}{\mu(T_1)}\int_{T_1} f(\cdot, \omega)d \mu, & \mbox{if $(t,
  \omega)\in T_1\times \Omega$.}
  \end{array}
  \right.
  \]
  One needs to verify $U_{T_1}(\omega, \hat f(t, \omega))=
  U_{T_1}(\omega, f(t, \omega))$ for all $(t,\omega)\in T_1 \times
  \Omega$. Suppose that there exist a coalition $D\subseteq T_1$
  and a state $\omega_0\in \Omega$ such that $U_{T_1}(\omega_0, \hat
  f(t, \omega_0))> U_{T_1}(\omega_0, f(t, \omega_0))$ for all $t\in
  D$. Then, a contradiction can be derived by a proof similar to that
  of Lemma \ref{lem:auxrob} except for the fact that the coalition
  $E_0$ can be chosen as $\bigcup_{\mathcal Q\in \mathfrak P(T_0)}
  E_0^\mathcal Q$, where each $E_0^\mathcal Q$ satisfies the
  condition
  \[
  \left\|\int_{E_0^\mathcal Q}(f(\cdot, \omega_0)- a(\cdot,
  \omega_0))d\mu- r_2 r_3\int_{T_0\cap T_\mathcal Q}(f(\cdot,
  \omega_0)- a(\cdot, \omega_0))d\mu\right\| <\frac{\epsilon}{
  |\mathfrak P(T_0)|},
  \]
  the blocking coalition is of the form $R= S\bigcup \left(T_0
  \setminus \bigcup_{\mathcal Q\in \mathfrak P(T_0)}T_\mathcal
  Q\right)$, where $S$ is defined in Lemma \ref{lem:auxrob},
  and the function $g:R\to Y_+$ is defined by
  \[
  g(t)= \left\{
  \begin{array}{ll}
  \hat f(t, \omega_0)+ \frac{u}{2\mu(E_0)}, & \mbox{if $t
  \in E_0$;}\\[0.5em]
  r_3 \hat f(t, \omega_0)+ \frac{u}{2\mu(C)}, & \mbox{if
  $t\in C$;}\\[0.5em]
  f(t, \omega_0), & \mbox{otherwise.}
  \end{array}
  \right.
  \]
  Note that $\bigvee \mathfrak P_R= \bigvee \mathfrak P_T=
  \mathcal F$ and $\int_R g d\mu \le \int_R a(\cdot, \omega_0)
  d\mu$. Let $b= \int_R a(\cdot, \omega_0)- \int_R g d\mu$.
  Consider a function $h: R\to Y_+$ defined by
  $h(t)= g(t)+ \frac{b}{\mu(R)}$. By (A$_3$), $U_t(\omega_0,
  h(t))> U_t(\omega_0, f(t, \omega_0))$ for almost all $t\in R$.
  Let $A(\omega_0)$ denote the atom of $\mathcal F$ containing
  $\omega_0$. Define a function $y:R\times \Omega\to Y_+$ by
  \[
  y(t, \omega)= \left\{
  \begin{array}{ll}
  h(t), & \mbox{if $(t, \omega)\in R\times
  A(\omega_0)$;}\\[0.5em]
  a(t, \omega), & \mbox{otherwise.}
  \end{array}
  \right.
  \]
  Then, $y$ is an $R$-allocation. Since $a(t, \cdot)$ is
  $\mathcal{F}$-measurable, $a(t, \omega)= a(t, \omega^\prime)$
  for almost all $t\in R$ and all $\omega, \omega^\prime \in
  A(\omega_0)$. Hence, $\int_{R}y(\cdot, \omega)d\mu=
  \int_{R}a(\cdot,\omega) d\mu$ for all $\omega\in A(\omega_0)$.
  By (A$_8$), $\bigvee \mathfrak P_R= \mathcal F$. Thus using
  (A$_9$), one has $\mathbb E_t[U_t(\cdot, f(t, \cdot))| \bigvee
  \mathfrak P_R]= U_t(\cdot, f(t, \cdot))$ and $\mathbb
  E_t[U_t(\cdot, y(t, \cdot))| \bigvee \mathfrak P_R]= U_t(\cdot,
  y(t, \cdot))$. Further, for all $\omega\in A(\omega_0)$ and
  almost all $t\in R$, one has that
  \begin{eqnarray*}
  \mathbb E_t\left[U_t(\cdot, y(t, \cdot))\big| \bigvee \mathfrak
  P_R\right](\omega)
  &=& U_t(\omega, y(t, \omega))= U_t(\omega_0, \hat h(t))\\
  &>& U_t(\omega_0, \hat f(t, \omega_0))\\ \noindent
  &=& \mathbb E_t\left[U_t(\cdot, f(t, \cdot))\big| \bigvee
  \mathfrak P_R\right](\omega),
  \end{eqnarray*}
  which implies that $f$ is $RW$-fine blocked by $R$ via $y$.
  This contradicts with the assumption. Hence,
  $U_{T_1}(\omega, f(t,\omega))\ge U_{T_1}(\omega, \hat f(t,
  \omega))$ for all $(t,\omega)\in T_1 \times \Omega$. By an
  argument similar to that in Lemma \ref{lem:auxrob}, one can
  further show $U_{T_1}(\omega,\hat f(t, \omega))= U_{T_1}
  (\omega, f(t,\omega))$ for all $(t, \omega)\in T_1 \times
  \Omega$.
  \end{proof}

  The following theorem is an extension of Theorem 3.1 in
  \cite{Einy-Moreno-Shitovitz:00} to mixed economies with
  infinitely many commodities and the exact feasibility. In
  addition, the assumption $\mathfrak P_T= \mathfrak P(T)$
  used by Einy et al. is not assumed in our result. To this
  end, we assume that for each $\omega\in \Omega$, $\mathcal
  E(\omega)$ denotes the symmetric information economy whose
  space of agents are $T$, and whose the consumption set, the
  utility function and the initial endowment of agent $t$ are
  $Y_+$, $U_t(\omega,\cdot)$ and $a(t, \omega)$ respectively.

  \begin{theorem}\label{thm:Ex-postCore}
  Assume that $\mathcal E$ satisfies \emph{(A$_1$)-(A$_9$)}. If
  $f$ is in the $RW$-fine core of $\mathcal E$, then it is also
  in the ex-post core of $\mathcal E.$
  \end{theorem}

  \begin{proof}
  Suppose that $f$ is not in the ex-post core of $\mathcal E$.
  The allocation $\hat f$ defined in Lemma \ref{lem:fineex} is
  not in the ex-post core of $\mathcal E$ either. Then there
  is a state $\omega_0\in \Omega$ such that $\hat f(\cdot,
  \omega_0)$ is a feasible allocation in the symmetric information
  economy $\mathcal E(\omega_0)$ and is not in the core of $\mathcal
  E(\omega_0)$. Consider an allocation $\hat f^\ast:T^\ast\times
  \Omega\to Y_+$ defined by $\hat f^\ast(t, \omega) = \hat f(t,
  \omega)$, if $(t, \omega)\in T_0\times \Omega$; and $\hat f^\ast(t,
  \omega) = \hat f(T_1, \omega)$, if $(t, \omega)\in T_1^\ast\times
  \Omega$, where $\hat f(T_1, \omega)$ denotes the constant value of
  $\hat f(\cdot,\omega)$ on $T_1$ for each $\omega\in \Omega$. Then
  $\hat f^\ast(\cdot,\omega_0)$ is a feasible allocation in $\mathcal
  E^\ast(\omega_0)$ and $\hat f^\ast(\cdot, \omega_0)$ is not in
  the core of $\mathcal E^\ast(\omega_0)$. Choose an arbitrary
  $A_0\in T_1$ and let $\mu(A_0)= \epsilon >0$. Note that under
  {\rm (A$_3$)}, the conclusion of Lemma \ref{lem:freeblock}
  also holds with the exact feasibility in a deterministic
  economy. Thus, $\hat f^\ast(\cdot, \omega_0)$ is blocked by a
  coalition $S^\ast$ via $\hat g^\ast$ such that $\mu^\ast(S^\ast)
  =\mu^\ast(T_0)+ \epsilon$, if
  \[
  \mu^\ast(T_1^\ast\setminus A_0^\ast)< \min \{\mu^\ast(T_0\cap
  T_{\mathcal Q}^\ast): \mathcal Q\in \mathfrak P(T_0)\},
  \]
  and otherwise,
  $\mu^\ast(S^\ast)> \mu^\ast(T_0) -\min\{\mu^\ast(T_0\cap
  T_{\mathcal Q}^\ast): \mathcal Q\in \mathfrak P(T_0)\}+ \
  \mu^\ast(T_1^\ast)$.
  Clearly, $\mu^\ast(S^\ast \cap T_1^\ast)\geq \epsilon$ and
  $\bigvee \mathfrak P(S^\ast)= \bigvee \mathfrak P(T)$. Let
  $\alpha= \frac{\epsilon}{\mu^\ast (S^\ast\cap T_1^\ast)}$.
  Applying (A$_3$) and an argument similar to that in Lemma
  \ref{lem:privateblocked}, one can show that there exists a
  coalition $R^\ast\subseteq \bigcup_{\mathcal Q\in \mathfrak
  P(T^\ast)}(S^\ast\cap T_\mathcal Q^\ast)$
  blocking $\hat f^\ast(\cdot, \omega_0)$ via $\hat
  h^\ast: R^\ast\to Y_+$ in
  $\mathcal E^\ast(\omega_0)$ such that $\mu^\ast(R^\ast\cap
  T_{\mathcal Q}^\ast)= \alpha \mu^\ast(S^\ast\cap T_{\mathcal
  Q}^\ast)$ for all $\mathcal Q\in \mathfrak P(S^\ast)$ and
  $\mu^\ast(R^\ast\cap T_1^\ast)= \epsilon$. Note that
  $\bigvee \mathfrak P(R^\ast)= \bigvee \mathfrak P(S^\ast)$.
  Consider a coalition $R$ of $\mathcal E$ define by $R= (R^\ast
  \cap T_0)\cup A_0$. Then, $\bigvee \mathfrak P(R)= \bigvee
  \mathfrak P(T)$. We consider a function $\hat h: R\to Y_+$
  defined by
  \[
  \hat h(t) = \left\{
  \begin{array}{ll}
  \hat h^\ast(t), & \mbox{if $t\in R^\ast\cap T_0$;}\\[0.5em]
  \frac{1}{\epsilon}\int_{R^\ast\cap T_1^\ast} \hat h^\ast
  d\mu^\ast, & \mbox{otherwise.}
  \end{array}
  \right.
  \]
  Obviously, $U_t(\omega_0, \hat h(t))> U_t (\omega_0, \hat f(t,
  \omega_0))$ if $t\in R^\ast\cap T_0$. By (A$_4$), $U_{T_1}
  (\omega_0,\hat h(t))> U_{T_1}(\omega_0, \hat f(t, \omega_0))$
  if $t= A_0$. Moreover, $\int_R \hat h d\mu = \int_R a(\cdot,
  \omega_0)d\mu$. Define a coalition $E= R\cup \left(T_0
  \setminus \bigcup_{\mathcal Q\in \mathfrak P(T_0)} T_\mathcal
  Q\right)$. Then $\bigvee \mathfrak P_T= \bigvee \mathfrak
  P_E$. Let $A(\omega_0)$ be the atom of $\bigvee \mathfrak
  P_T$ containing $\omega_0$. Now, define a function $y:E
  \times \Omega\to Y_+$ such that
  \[
  y(t, \omega) = \left\{
  \begin{array}{ll}
  \hat h(t), & \mbox{if $(t, \omega)\in R\times A(\omega_0)$;}
  \\[0.5em]
  a(t, \omega), & \mbox{otherwise.}
  \end{array}
  \right.
  \]
  Then, $y$ is an $E$-allocation. Applying an argument similar
  to that in Lemma \ref{lem:fineex}, one can show that $f$
  is $RW$-fine blocked by $E$ via $y$. This contradicts with
  the assumption, which completes the proof.
  \end{proof}

  \begin{remark}
  It is obvious from the proof of Theorem \ref{thm:Ex-postCore}
  that a similar result holds for atomless economies under
  (A$_1$)-(A$_3$), (A$_5$) and (A$_8$)-(A$_9$) only.
  \end{remark}

  \subsection{The weak fine core}
  \label{subsec:weakfine}

  In this subsection, we extend Proposition 5.1 in
  \cite{Einy-Moreno-Shitovitz:01} to mixed economies with infinitely
  many commodities and the exact feasibility. We also relax the
  assumption $\mathfrak P_T= \mathfrak P(T)$.

  \begin{definition}
  A feasible assignment $f$ in $\mathcal E$ is said to be in the
  \emph{weak fine core} of $\mathcal E$ if $f(t, \cdot)$ is $\bigvee
  \mathfrak P_T$-measurable for almost all $t\in T$, and $f$ cannot
  be $NY$-fine blocked by any coalition.
  \end{definition}

  In the sequel, the economy $\mathcal E^s$ is similar to
  $\mathcal{E}$ except for the information of every agent being
  $\bigvee \mathfrak P_T$. The proof of the next lemma is similar
  to that of Lemma \ref{lem:fineex}.

  \begin{lemma}\label{lem:weakpri}
  Assume that $f$ is in the weak fine core of $\mathcal E$.
  Under \emph{(A$_1$)-(A$_7$)}, there exists an allocation $\hat f$
  such that $\hat{f}|_{T_0\times \Omega}=f$, $\hat f(\cdot, \omega)$
  is constant on $T_1$ for each $\omega\in \Omega$, $V_{T_1}(\hat
  f(t, \cdot))= V_{T_1}(f(t, \cdot))$ for almost all $t\in T_1$ and
  $\int_T f d\mu= \int_T \hat f d\mu$.
  \end{lemma}

  \begin{theorem}\label{thm:weakfine}
  Assume that $\mathcal{E}$ satisfies \emph{(A$_1$)-(A$_7$)}. Then
  $f$ is in the weak fine core of $\mathcal E$ if and only if $f$
  is in the private core of $\mathcal E^s$.
  \end{theorem}

  \begin{proof}
  It is clear that if $f$ is in the private core of $\mathcal E^s$,
  then $f$ is in the weak fine core of $\mathcal E$. Now, assume
  that $f$ is in the weak fine core of $\mathcal E$. Let $\bigvee
  \mathfrak P_T$ be generated by the partition $\{A_1, ...,A_k\}$
  of $\Omega$, and let $X$ denote the set of all $\bigvee \mathfrak
  P_T$-measurable elements of $(Y_+)^\Omega$. Define a function
  $\gamma:X\to Y_+^k$ such that $\gamma(f)= f_s$, where $f_s=
  (f(\omega_1), ...,f(\omega_k))$ if $\omega_j \in A_j$ for all
  $1\leq j\leq k$. Now for all $t\in T$, consider a function
  $V_t^s:Y_+^k \to \mathbb R$ defined by $V_t^s(f_s)=
  V_t(\gamma^{-1}(f_s))$. Let $\widetilde{\mathcal E}^s$ be a
  symmetric information economy whose space of economic agents is
  $(T, \Sigma, \mu)$, and in which the consumption set of every agent
  is $Y_+^k$, the utility function and initial endowment of agent
  $t$ are $V_t^s$ and $a^s(t)= \gamma(a(t, \cdot))$ respectively.
  Suppose that $f$ is not in the private core of $\mathcal E^s$.
  Then $f_s$ is not in the private core of $\widetilde{\mathcal
  E}^s$. Thus $\hat f_s$ is is not in the private core of
  $\widetilde{\mathcal E}^s$. Applying an argument similar to
  that in the proof of Theorem \ref{thm:Ex-postCore}, one can
  show that there is a coalition $R\subseteq \bigcup_{\mathcal
  Q\in \mathfrak P(T)} T_\mathcal Q$ blocking $f^s$ via $h^s$
  such that $\bigvee \mathfrak P(R)= \bigvee \mathfrak P(T)$. Let
  $E= R\cup \left(T_0\setminus \bigcup_{\mathcal Q\in \mathfrak
  P(T_0)} T_\mathcal Q\right)$. Obviously, $\bigvee \mathfrak P_T
  = \bigvee \mathfrak P_E$. Define a function $y^s:E\to Y_+^k$ by
  \[
  y^s(t) := \left\{
  \begin{array}{ll}
  h^s(t), & \mbox{if $t\in R$;}\\[0.5em]
  a^s(t), & \mbox{otherwise.}
  \end{array}
  \right.
  \]
  Note that $V_t^s(y^s(t))>V_t^s(f^s(t))$ for almost all $t\in E$.
  Furthermore, $\int_E y^s d\mu=\int_E a^s d\mu$. Let $y(t,\cdot)
  =\gamma^{-1}(y^s(t))$ for all $t\in E$. Then, $y(t, \cdot)$ is
  $\bigvee \mathfrak P_E$-measurable and $V_t(y(t, \cdot))>
  V_t (f(t, \cdot))$ for almost all $t\in E$. Moreover, $\int_E
  y(\cdot, \omega)d\mu= \int_E a(\cdot, \omega)d\mu$ for all
  $\omega\in \Omega$. Thus, $f$ is also $NY$-fine blocked by $E$
  via $y$. This contradicts with the fact that $f$ is in the weak
  fine core of $\mathcal E$. Consequently, $f$ must be in the
  private core of $\mathcal E^s$.
  \end{proof}

  \begin{remark}
  A similar conclusion can be derived for atomless economies
  under (A$_1$)-(A$_3$) and (A$_5$) only.
  \end{remark}

%
%


\bibliographystyle{spbasic}      


\end{document}